\documentclass{amsart} 
\usepackage{hyperref}
\usepackage{enumerate} \usepackage{upref}
\usepackage{verbatim} \usepackage{color}
\usepackage{graphicx}
\usepackage[all]{xy}

\newcommand{\arxiv}[1]{\href{http://arxiv.org/pdf/#1}{arXiv:#1}}
\newcommand*{\mailto}[1]{\href{mailto:#1}{\nolinkurl{#1}}}

\newtheorem{theorem}{Theorem}[section]
\newtheorem{lemma}[theorem]{Lemma}
\newtheorem{corollary}[theorem]{Corollary}
\newtheorem{remark}[theorem]{Remark}

\newtheorem{proposition}[theorem]{Proposition}

\numberwithin{equation}{section}
\newtheorem{definition}[theorem]{Definition}
\allowdisplaybreaks

\newcommand{\epsi}{\varepsilon}
\newcommand{\Winf}{{W^{1,\infty}(\Real)}}
\newcommand{\Gr}{G}

\newcommand{\D}{\ensuremath{\mathcal{D}}}

\newcommand{\G}{\ensuremath{\mathcal{G}}}
\newcommand{\F}{\ensuremath{\mathcal{F}}}

\newcommand{\Real}{\mathbb{R}}

\newcommand{\muac}{\mu_{\text{\rm ac}}}
\newcommand{\muacZ}{\mu_{0, \text{\rm ac}}}

\newcommand{\quot}{{\F/\Gr}}
\newcommand{\inv}{{^{-1}}}
\newcommand{\dott}{\, \cdot\,}
\newcommand{\dx}{\,dx}

\newcommand{\sign}{\mathop{\rm sign}}

\newcommand{\nn}{\nonumber}

\newcommand{\norm}[1]{\left\Vert#1\right\Vert}
\newcommand{\abs}[1]{\left\vert#1\right\vert}

\newcommand{\Linf}{{L^\infty(\Real)}}
\newcommand{\Ltwo}{{L^2(\Real)}}

\DeclareMathOperator{\meas}{meas}

\DeclareMathOperator{\id}{Id}

\def\XXint#1#2#3{{\setbox0=\hbox{$#1{#2#3}{\int}$}
    \vcenter{\hbox{$#2#3$}}\kern-.5\wd0}}

\numberwithin{equation}{section}

\begin{document}

\title[Global solutions for the two-component
Camassa--Holm system]{Global solutions for the
  two-component Camassa--Holm system}

\author[K. Grunert]{Katrin Grunert}
\address{Department of Mathematical Sciences\\
  Norwegian University of Science and Technology\\
  7491 Trondheim\\ Norway}
\email{\mailto{katring@math.ntnu.no}}
\urladdr{\url{http://www.math.ntnu.no/~katring/}}

\author[H. Holden]{Helge Holden}
\address{Department of Mathematical Sciences\\
  Norwegian University of Science and Technology\\
  7491 Trondheim\\ Norway\\ {\rm and} Centre of
  Mathematics for Applications\\ University Oslo\\
  0316 Oslo\\ Norway}
\email{\mailto{holden@math.ntnu.no}}
\urladdr{\url{http://www.math.ntnu.no/~holden/}}

\author[X. Raynaud]{Xavier Raynaud}
\address{Centre of Mathematics for Applications\\
  University Oslo\\ 0316 Oslo\\ Norway}
\email{\mailto{xavierra@cma.uio.no}}
\urladdr{\url{http://www.folk.uio.no/xavierra/}}

\date{\today} 
\thanks{Research supported in part by the
  Research Council of Norway under projects Wavemaker, NoPiMa, and by the Austrian Science Fund (FWF) under Grant No.~J3147.}  

\begin{abstract}
  We prove existence of a global conservative solution of the Cauchy problem for the two-component Camassa--Holm (2CH) system on the line, allowing for nonvanishing and distinct asymptotics at plus and minus infinity. The solution is proven to be smooth as long as the density is bounded away from zero. Furthermore, we show that by taking the limit of vanishing density in the 2CH system, we obtain the global conservative solution of the (scalar) Camassa--Holm equation, which provides a novel way to define and obtain these solutions. Finally, it is shown that while solutions of the 2CH system have infinite speed of propagation, singularities travel with finite speed.
\end{abstract}
\maketitle

\section{Introduction}

The two-component Camassa--Holm (2CH) system, which was first derived in \cite[Eq.~(43)]{OlverRosenau}, is given by
\begin{subequations}
  \label{eq:chsys}
  \begin{align}
    \label{eq:chsys11}
    u_t-u_{txx}+\kappa u_x+3uu_x-2u_xu_{xx}-uu_{xxx}+\eta\rho\rho_x&=0,\\
    \label{eq:chsys12}
    \rho_t+(u\rho)_x&=0,
  \end{align}
\end{subequations}
for constants $\kappa\in\Real$ and $\eta\in(0,\infty)$, or equivalently
\begin{subequations}
 \label{eq:rewchsys10}
\begin{align}
 \label{eq:rewchsys11}
u_t+uu_x+P_x&=0,\\
\label{eq:rewchsys12}
    \rho_t+(u\rho)_x&=0,
\end{align}
\end{subequations}
  where $P$ is implicitly defined by
\begin{equation}
  \label{eq:rewchsys13}
  P-P_{xx}=u^2+\kappa u+\frac12u_x^2+\eta\frac12\rho^2.
\end{equation}
We here study the Cauchy problem on the line where the equations  \eqref{eq:rewchsys10}--\eqref{eq:rewchsys13} are augmented with initial conditions $(u,\rho)|_{t=0}=(u_0,\rho_0)$.
The equations have been derived as a model for shallow water by Constantin and Ivanov \cite{MR2474608}, where it is shown that $\eta$ positive and  $\rho$  nonnegative is the physically relevant case.  

The purpose of this paper is twofold: First of all we want to show the existence of a global and conservative solution of the 2CH system by suitable modifying recent results \cite{GHR:12} for the (scalar) Camassa--Holm (CH) equation 
\begin{equation}
  \label{eq:ch}
    u_t-u_{txx}+\kappa u_x+3uu_x-2u_xu_{xx}-uu_{xxx}=0,
\end{equation}
(which one obtains by taking $\rho$ identically zero in  \eqref{eq:rewchsys10}--\eqref{eq:rewchsys13}, or  \eqref{eq:chsys}). It turns out that the solution of  \eqref{eq:rewchsys10}--\eqref{eq:rewchsys13} is regular as long as the initial density $\rho_0$ is bounded away from zero (Theorem \ref{th:presreg} and Corollary \ref{cor:presreg3}).  This contrasts the case of the CH equation where one in general encounters weak solutions only.  Secondly, we study the limit of the global conservative solution of the 2CH system for a sequence $\rho_0^n$ of  initial densities tending to zero as $n\to\infty$. We find  that the solution of the 2CH system  approaches the global conservative solution of the Camassa--Holm equation (Theorem \ref{th:approxCH}). This offers an alternative approach to the study of conservative solutions, and a novel way to define conservative solutions. This is interesting because  the CH equation enjoys two distinct classes of solutions, denoted dissipative and conservative solutions, respectively. In brief terms, the conservative solution preserves energy, while energy decreases for dissipative solutions. Both classes will in general  have weak solutions rather than smooth solutions. To identify and characterize the two classes has turned out to be rather involved, see, e.g., \cite{BreCons:07,BreCons:09,HolRay:07,HolRay:09,GHRb:10,GHR:11}, and references therein. The  approach in this paper characterizes conservative solutions as limits of smooth (classical) solutions of the 2CH system. This is novel.

The 2CH system
\eqref{eq:rewchsys10}--\eqref{eq:rewchsys13} has
been studied extensively, from many different
points of view, making a complete list of
references too long. However, we here mention that
Wang, Huang, and Chen \cite{WangHuangChen} have
studied conservative and global solutions of the
2CH system using a change of variables similar to
the one employed here. The results here are more
detailed and precise. In particular, we establish
the semigroup property of the solutions and the
continuity of the semigroup with respect to a new
distance which is introduced.  The vanishing
density limit is not discussed in \cite{WangHuangChen}.  
Escher, Lechtenfeld, and Yin  \cite{eschlechyin:07} established 
a short-time existence theory for  solutions using Kato techniques in
the case $\eta=-1$ and $\kappa=0$. 
In the same paper it is
shown that solutions may blow up in final
time. Our approach does not apply to the case with $\eta$ negative. Constantin and Ivanov \cite{MR2474608} showed
that the solution for small initial data (or, more
precisely, for $\rho_0$ close to a constant and
small $u_0$) remains smooth. We here extend this result to
data of arbitrary size, provided the density is
bounded away from zero initially. A remarkable
property of the system, which is shown in
\cite{GuanYin2010a}, is that when $\rho_0(x)>0$
for all $x\in\Real$, the solution exists globally
in time. In Theorem \ref{th:presreg} we 
establish a local smoothing effect of the variable
$\rho$, thereby extending the result of \cite{GuanYin2010a}. In particular, we show how the characteristics govern the domain of smoothness.

For other related results pertaining to the
present system, please see
\cite{GuanYin2010a,GuiLiu2011,GuiLiu2010}.  In
addition to the 2CH system discussed in the
present paper, there exists several other
two-component generalizations of the CH equation,
see, e.g., \cite{ChenLiu2010, FuQu:09,
  GuanKarlsenYin2010,GuanYin2010,GuoZhou2010,
  Kuzmin, TanYin}. For traveling wave solutions
see \cite{MR2474608,ChenLiuZhang,Mohajer,Mustafa}.

Let us next describe the content of this paper more precisely.  While we in this paper treat the case of arbitrary (nonvanishing) asymptotics of the initial data (and thereby of the solution), in the sense that
 \begin{equation*}
 \lim_{x\to\pm\infty} u_0(x)=u_{\pm\infty}\quad\text{ and }\quad\lim_{\abs{x}\to\infty} \rho_0(x)=\rho_{\infty},
\end{equation*} 
we here, in order to make the presentation in the introduction more transparent,   assume vanishing asymptotics for $u$, that is, $u_{\pm\infty}=0$. Furthermore, we assume  
$\eta=1$ and 
$\kappa=0$.  We first make a change from Eulerian to Lagrangian variables and
introduce a new energy variable. The change of variables, which we now will detail,  is related to the one used in \cite{HolRay:07} and, in particular, \cite{GHR:12}.  Assume that $u=u(x,t)$ is a solution, and define the characteristics $y=y(t,\xi)$ by
\begin{equation*}
y_t(t,\xi)=u(t,y(t,\xi))
\end{equation*}
and the Lagrangian velocity by
\begin{equation*}
  U(t,\xi)=u(t,y(t,\xi)).
\end{equation*} 
By introducing the Lagrangian energy density $h$ and density $r$ by
\begin{equation*}
 h(t,\xi)=u_x^2(t,y(t,\xi))y_\xi(t,\xi)+\rho^2(t,y(t,\xi))y_\xi(t,\xi), \quad  r(t,\xi)=\rho(t,y(t,\xi))y_\xi(t,\xi),
\end{equation*}
we find that the system can be rewritten as
\begin{align*}
   y_t&=U,\\
    U_t&=-Q,\\
    h_t&=2(U^2-P)U_\xi,\\
    r_t&=0,
\end{align*}
where the functions $P$ and $Q$ are explicitly given by \eqref{eq:P} and  \eqref{eq:Q}, respectively.  We then establish the existence of a unique global solution for this system (see Theorem \ref{th:global}) which forms a continuous semigroup in an appropriate norm.  
In order to solve the Cauchy problem we have to choose the initial data appropriately. To accommodate for the possible concentration of energy we augment the natural initial data $u_0$ and $\rho_0$ with a nonnegative Radon measure $\mu_0$ such that the absolutely continuous part $\muacZ$   is $\muacZ=(u_{0,x}^2+(\rho_0-\rho_{0,\infty})^2)\,dx$.  The precise translation of these initial data is given in Theorem \ref{th:Ldef}.  One then solves the system in Lagrangian coordinates. The translation back to Eulerian variables is described in Theorem \ref{th:umudef}. However, there is an intrinsic problem in this latter translation if one wants a continuous semigroup. This is due to the problem of relabeling; to each solution in Eulerian variables there exist several distinct solutions in Lagrangian variables as there are additional degrees of freedom in the Lagrangian variables. In order to resolve this issue to get a continuous semigroup, one has to identify Lagrangian  functions corresponding to one and the same Eulerian solution. This is treated in 
Sec.~\ref{sec:eulerlagrange}.  The main existence theorem, Theorem \ref{th:mainX}, states that for  $u_0\in H^1$ and $\rho_0\in L^2$ and $\mu_0$ a nonnegative Radon measure with absolutely continuous part $\muacZ$   such that $\muacZ=(u_{0,x}^2+(\rho_0-\rho_{0,\infty})^2)\,dx$, there exists a continuous semigroup $T_t$ such that $(u,\rho,\mu)(t)=T_t(u_0,\rho_0,\mu_0)$ is a weak global and conservative solution of the 2CH system. In addition, the measure $\mu$ satisfies
\begin{equation*}
  (u^2+\mu)_t+(u(u^2+\mu))_x=(u^3-2Pu)_x,
\end{equation*}
weakly. Furthermore, for almost all times the measure $\mu$ is absolutely continuous and $\mu=(u_{x}^2+(\rho-\rho_{\infty})^2)\,dx$.

In order to analyze the case where we consider a sequence of initial densities $\rho_0^n$ tending to zero as $n\to\infty$, we need to have a sufficiently strong stability result. To that effect, we have the following result, Theorem \ref{th:approxCH}.  Consider a sequence of initial data $(u_0^n,\rho_0^n, \mu_0^n)$ such that 
$u_0^n\to u_0$ in $H^1(\Real)$, $\rho_0^n-\rho_{0,\infty}^n \to 0$ in $L^2(\Real)$,  
$\rho_{0,\infty}^n \to 0$, with $\rho_0^n\ge d_n>0$ for all $n$. Assume that the initial measure is absolutely continuous, that is,  
$\mu_0^n=\muacZ^n=((u_{0,x}^n)^2+(\rho_0^n- \rho_{0,\infty}^n)^2)\,dx$. Then the sequence $u^n(t)$ will converge in $L^\infty(\Real)$ to the weak, conservative global solution of the Camassa--Holm equation with initial data $u_0$.

Finally, we want to address the regularity issue. Consider an open subset I of $\Real$. We say that $(u,\rho,\mu)$ is $p$-regular on  $I$ if
\begin{equation*}
  u\in W^{p,\infty}(I),\ \rho\in W^{p-1,\infty}(I)\ \text{ and } \muac=\mu\text{ on }I.
\end{equation*}
A surprising feature of the 2CH system is that while it has an infinite speed of propagation \cite{henry:09}, singularities travel with finite speed. This is the content of the following theorem, Theorem \ref{th:presreg}.  Assume that the initial data $(u_0,\rho_0,\mu_0)$ is $p$-regular on  an interval $(x_0,x_1)$ such that  $\rho_0^2\ge c>0$ on  $(x_0,x_1)$.  Then the solution $(u,\rho,\mu)(t)$ is $p$-regular on the interval given by the characteristics emanating from  $(x_0,x_1)$. More precisely, it is 
$p$-regular on the interval $(y(t,\xi_0),
  y(t,\xi_1))$, where $\xi_0$ and $\xi_1$ satisfy
  $y(0,\xi_0)=x_0$ and $y(0,\xi_1)=x_1$ and are
  defined as
  \begin{equation*}
    \xi_0=\sup\{\xi\in\Real\ |\ y(0,\xi)\leq x_0\} \text{ and }
    \xi_1=\inf\{\xi\in\Real\ |\ y(0,\xi)\geq x_1\}.
  \end{equation*}
Thus we see that regularity is preserved between characteristics.
 
\section{Eulerian coordinates}

We consider the Cauchy problem for the two-component Camassa--Holm system with arbitrary 
$\kappa\in\Real$ and $\eta\in(0,\infty)$, 
\begin{subequations}
  \label{eq:chsys2}
  \begin{align}
    \label{eq:chsys21}
    u_t-u_{txx}+\kappa u_x+3uu_x-2u_xu_{xx}-uu_{xxx}+\eta\rho\rho_x&=0,\\
    \label{eq:chsys22}
    \rho_t+(u\rho)_x&=0,
  \end{align}
\end{subequations}
with initial data $u|_{t=0}=u_0$ and $\rho|_{t=0}=\rho_0$.
We are interested in global solutions for initial data $u_0$ with nonvanishing  and possibly distinct limits at infinity, that is, 
\begin{equation}\label{eq:nonvanlim}
 \lim_{x\to-\infty} u_0(x)=u_{-\infty}\quad\text{ and }\quad\lim_{x\to\infty} u_0(x)=u_{\infty}.
\end{equation}
Furthermore we assume that the initial density has equal asymptotics which need not to be zero, that is,
\begin{equation}\label{eq:nonvanlimR}
 \lim_{x\to\pm\infty} \rho_0(x)=\rho_{\infty}.
\end{equation}
More precisely we introduce the spaces
\begin{equation}
H_{\infty}(\Real)=\{v\in L^1_{\rm loc}(\Real) \mid 
v(x)=\bar v(x)+v_{-\infty}\chi(-x)+v_\infty \chi(x), \, \bar v\in H^1(\Real), v_{\pm\infty}\in\Real\}, \label{eq:Hinf} 
\end{equation}
where $\chi$ denotes a smooth partition function 
with support in $[0,\infty)$ such that $\chi(x)=1$ for $x\geq 1$ and
$\chi'(x)\geq 0$ for $x\in\Real$, and
\begin{equation}
L^2_{\rm const}(\Real)= \{g\in L^1_{\rm loc}(\Real) \mid g(x)=g_\infty+\bar g(x), \, \bar g\in L^2(\Real),  g_\infty\in \Real\}.\label{eq:Lconst} 
\end{equation}
Subsequently, we will assume that
\begin{equation}
u_0\in H_{\infty}(\Real), \quad \rho_0\in L^2_{\rm const}(\Real).
\end{equation}

Introducing the
mapping $I_{\chi}$ from $H^1(\Real)\times\Real^2$
into $H_{\text{loc}}^1(\Real)$ given by
\begin{equation*}
  I_{\chi}(\bar u,c_-,c_+)(x)=\bar u(x)+c_-\chi(-x)+c_+\chi(x)
\end{equation*}
for any $(\bar u,c_-,c_+)\in
H^1(\Real)\times\Real^2$, yields that any initial
condition $u_0\in H_\infty(\Real)$ is defined by
an element in $H^1(\Real)\times \Real^2$ through
the mapping $I_\chi$.  Hence we see that
$H_{\infty}(\Real)$ is the image of
$H^1(\Real)\times\Real^2$ by $I_\chi$, that is,
$H_{\infty}(\Real)=I_{\chi}(H^1(\Real)\times\Real^2)$.The
linear mapping $I_\chi$ is injective. We equip
$H_\infty(\Real)$ with the norm
\begin{equation}
  \label{eq:normHinfty}
  \norm{u}_{H_\infty}=\norm{\bar u}_{H^1}+\abs{c_-}+\abs{c_+}
\end{equation}
where $u=I_\chi(\bar u,c_-,c_+)$.  Then
$H_\infty(\Real)$ is a Banach space. Given another
partition function $\tilde\chi$, we define the
mapping $(\tilde{\bar u},\tilde c_-,\tilde
c_+)=\Psi(\bar u,c_-,c_+)$ from $H^1(\Real)\times\Real^2$
to $H^1(\Real)\times\Real^2$ as $\tilde c_-=c_-$, $\tilde
c_+=c_+$ and
\begin{equation}
  \label{eq:defPsi}
  \tilde{\bar u}(x)=\bar u(x)+c_-(\chi(-x)-\tilde\chi(-x))+c_+(\chi(x)-\tilde\chi(x)).
\end{equation}
The linear mapping $\Psi$ is a continuous
bijection. Since
\begin{equation*}
  I_{\chi}=I_{\tilde\chi}\circ \Psi,
\end{equation*}
we can see that the definition of the Banach space
$H_\infty(\Real)$ does not depend on the choice of
the partition function $\chi$. The norm defined by
\eqref{eq:normHinfty} for different partition
functions $\chi$ are all equivalent.

Similarly, one can associate to any element $\rho\in L^2_{\rm const}(\Real)$ the unique pair $(\bar\rho, k)\in L^2(\Real)\times \Real$ through the mapping $J$ from $L^2(\Real)\times\Real$ to $L^2_{\rm const}(\Real)$ which is defined as 
\begin{equation}
 J(\bar\rho,k)=\bar\rho+k.
\end{equation}
In fact $J$ is bijective from $L^2(\Real)\times\Real$ to $L^2_{\rm const}(\Real)$, which allows us to equip $L^2_{\rm const}(\Real)$ with the norm 
\begin{equation}\label{normL2inf}
 \norm{\rho}_{L^2_{\rm const}}=\norm{\bar\rho}_{L^2}+\vert k\vert,
\end{equation}
where we decomposed $\rho$ according to $\rho=J(\bar\rho,k)$. Thus $L^2_{\rm const}(\Real)$ together with the norm defined in \eqref{normL2inf} is a Banach space.

Note that for smooth solutions, we have the
following conservation law
\begin{equation}
  \label{eq:conslaw}
  (u^2+u_x^2+\eta\rho^2)_t+(u(u^2+u_x^2+\eta\rho^2))_x=(u^3+\kappa u^2-2Pu)_x.
\end{equation}

Moreover, if $(u(t,x),\rho(t,x))$ is a solutions of the two-component Camassa--Holm system \eqref{eq:chsys2}, then, for any constant $\alpha\in\Real$ we easily find that 
\begin{equation}
 v(t,x)=u(t,x-\alpha t)+\alpha,\quad \text{ and }\quad \tau(t,x)=\sqrt{\eta}\rho(t,x-\alpha t), 
\end{equation}
solves the two-component Camassa--Holm system with
$\kappa$ replaced by $\kappa-2\alpha$ and
$\eta=1$. Therefore, without loss of generality,
we assume in what follows, that $\lim_{x\to
  -\infty}u_0(x)=0$ and $\eta=1$. In addition, we
only consider the case $\kappa=0$ as one can make
the same conclusions for $\kappa\not =0$ with
slight modifications.

\section{Global solutions in Lagrangian coordinates}

The aim of this section is to rewrite the two-component Camassa--Holm system as a system of ordinary differential equations in a suitable Banach space, such that we can prove global existence of solutions therefore.

\subsection{Equivalent system}

Rewrite the two-component Camassa--Holm system
as the following system
\begin{subequations}
 \label{eq:rewchsys}
\begin{align}
 \label{eq:rewchsys1}
u_t+uu_x+P_x&=0,\\
\label{eq:rewchsys2}
    \rho_t+(u\rho)_x&=0,
\end{align}
\end{subequations}
  where $P$ is implicitly defined by\footnote{For $\kappa$ nonzero \eqref{eq:rewchsys3}
  is simply replaced by $P-P_{xx}=u^2+\kappa u
  +\frac{1}{2}u_x^2+\frac{1}{2}\rho^2$.}
\begin{equation}
  \label{eq:rewchsys3}
  P-P_{xx}=u^2+\frac12u_x^2+\frac12\rho^2.
\end{equation}
Introduce the subspace $H_{0,\infty}(\Real)$ of  $H_\infty(\Real)$ as 
\begin{equation*}
  H_{0,\infty}(\Real)=I_{\chi}(H^1(\Real)\times\{0\}\times\Real).
\end{equation*}
Then we obtain, using \eqref{eq:rewchsys3}, under the assumption that $\rho\in L^2_{\rm const}(\Real)$ and $u\in H_{0,\infty}(\Real)$
that the function $P$ can be expressed through
\begin{align}\label{rep:p}
 P(x)&=c^2\chi^2(x)+\frac{1}{2}\int_\Real e^{-\vert x-z\vert } (2c\chi\bar u +\bar u^2+\frac{1}{2}u_x^2+2c^2\chi^{\prime 2}+2c^2\chi\chi^{\prime\prime})(z) dz\\ \nn
&\quad + \frac{1}{2} k^2+\frac{1}{2}\int_\Real e^{-\vert x-z\vert } (\frac{1}{2}\bar \rho^2+k\bar \rho)(z)dz .
\end{align}
In particular, $P\in H_{\infty}(\Real)$ and
especially $P_x\in L^2(\Real)$.

We introduce the Lagrangian variables and rewrite
the governing equations \eqref{eq:rewchsys} with
respect to these variables.  Let the
characteristics $y(t,\xi)$, which can be decomposed as $y(t,\xi)=\zeta(t,\xi)+\xi$, be defined as the
solutions of
\begin{equation}
\label{eq:char}
y_t(t,\xi)=u(t,y(t,\xi))
\end{equation}
for some given initial function $y(0,\xi)$. We define
\begin{equation}
  \label{eq:defU}
  U(t,\xi)=u(t,y(t,\xi)),
\end{equation}
which can be decomposed as
\begin{equation}
 U(t,\xi)=\bar U(t,\xi)+c(t)\chi\circ y(t,\xi),
\end{equation}
where $\bar U\in H^1(\Real)$ and $c\in\Real$. 
In addition, we define $h\in L^2(\Real)$ formally as 
\begin{equation}
 h(t,\xi)=u_x^2(t,y(t,\xi))y_\xi(t,\xi)+\bar \rho^2(t,y(t,\xi))y_\xi(t,\xi),
\end{equation}
so that $(u_x^2(t,x)+\bar \rho^2(t,x))dx=y_\# (h(t,\xi)d\xi)$, and $r\in L^2_{\rm const}(\Real)$ as
\begin{equation}\label{eq:rhotor}
 r(t,\xi)=\rho(t,y(t,\xi))y_\xi(t,\xi).
\end{equation}
Moreover, $r$ can be decomposed as
\begin{equation}
r(t,\xi)=\bar r(t,\xi)+k(t)y_\xi(t,\xi),                                    
\end{equation}
so that $\rho(t,x)dx=k(t)+y_\# (\bar
r(t,\xi)d\xi)$.  Here $y_\#$ denotes the
pushforward.\footnote{The push-forward of a
  measure $\nu$ by a measurable function $f$ is
  the measure $f_\#\nu$ defined as
  $f_\#\nu(B)=\nu(f^{-1}(B))$ for any Borel set
  $B$.}  We have
\begin{equation}
 y_\xi h= U_\xi^2+\bar r^2.
\end{equation}
The system
\eqref{eq:rewchsys}--\eqref{eq:rewchsys3} rewrites
\begin{subequations}
  \label{eq:chsyseq0}
  \begin{align}
    \label{eq:chsyseq1}
    y_t&=U,\\
    \label{eq:chsyseq2}
    U_t&=-Q,\\
    \label{eq:chsyseq40}
    c_t&=0,\\
    \label{eq:chsyseq50}
    h_t&=2(U^2+\frac{1}{2}k^2-P)U_\xi,\\
    \label{eq:chsyseq60}
    r_t&=0,\\
    \label{eq:chsyseq80}
    k_t&=0,
  \end{align}
\end{subequations}
where the functions $P(t,\xi)$ and $Q(t,\xi)$ are given by
\begin{equation}
 \label{eq:P}
P(\xi)=\frac14\int_\Real
  e^{-\abs{y(\xi)-y(\eta)}}((2\bar U^2+4c\bar
  U\chi\circ y)y_\xi+2k\bar r +h)(\eta)\,d\eta+c^2g\circ y(\xi)+\frac{1}{2}k^2,
\end{equation}
and 
\begin{equation}
 \label{eq:Q}
Q(\xi)=-\frac14\int_\Real \sign(\xi-\eta)e^{-\abs{y(\xi)-y(\eta)}}((2\bar
  U^2+4c\bar U\chi\circ y)y_\xi+2k\bar r+h)(\eta)\,d\eta+c^2 g'\circ y(\xi),
\end{equation}
where 
\begin{equation}\label{def:g}
 g(x)=\chi^2(x)+\frac12\int_\Real e^{-\abs{x-z}}(2\chi^{\prime 2}+2\chi\chi^{\prime\prime})(z) \,dz.
\end{equation}
Equations \eqref{eq:chsyseq1},
\eqref{eq:chsyseq2}, \eqref{eq:chsyseq50}, and
\eqref{eq:chsyseq60}, follow directly from
\eqref{eq:rewchsys}--\eqref{eq:rewchsys3}.  The
value of the energy density $(u^2+u_x^2+\rho^2exit)$ is
essentially controlled by the variable $h$, and the
evolution equation for $h$, \eqref{eq:chsyseq50},
follows from the conservation law
\eqref{eq:conslaw}, which yields
\begin{equation*}
  (u_x^2+\rho^2)_t+(u(u_x^2+\rho^2))_x=2(u^2-P)u_x.
\end{equation*}
Equations \eqref{eq:chsyseq40} and
\eqref{eq:chsyseq80} state that the asymptotics do
not evolve in time. Formally, it comes from the
fact that, if all functions are smooth and the
derivatives decay sufficiently fast, then
\begin{equation*}
 \lim_{x\to\pm\infty}u_t(x)=- \lim_{x\to\pm\infty}(uu_x+ P_x)=0,
\end{equation*}
and 
\begin{equation*}
    \lim_{x\to\pm\infty}\rho_t(x)=-\lim_{x\to\pm\infty}{(\rho u_x+u\rho_x)}=0.
  \end{equation*}

The variables $y$, $U$, and $r$ given by
\eqref{eq:char}, \eqref{eq:defU} and
\eqref{eq:rhotor} can be seen as the standard
Lagrangian variables for the equation. However, to
carry out our analysis we need to introduce the
variables $\zeta$, $\bar U$ and $\bar r$, which
are derived from the former variables and, in
addition, belong to the appropriate Banach
spaces. In those variables the system of governing
equation writes
\begin{subequations}
  \label{eq:chsyseq}
  \begin{align}
    \label{eq:chsyseq11}
    \zeta_t&=U,\\
    \label{eq:chsyseq3}
    \bar U_t&=-Q-c(\chi^\prime\circ y)U,\\
    \label{eq:chsyseq5}
    h_t&=2(U^2+\frac{1}{2}k^2-P)U_\xi,\\
    \label{eq:chsyseq7}
    \bar r_t&= -kU_\xi,\\
    \label{eq:chsyseq4}
    c_t&=0,\\
    \label{eq:chsyseq8}
    k_t&=0.
  \end{align}
\end{subequations}

\subsection{Existence and uniqueness of solutions of the equivalent system}

Let $V$ be the Banach space defined by
\begin{equation*}
V=\{f\in C_b(\Real) \ |\ f_\xi\in\Ltwo\}
\end{equation*}
where $C_b(\Real)=C(\Real)\cap\Linf$ and the norm
of $V$ is given by
$\norm{f}_V=\norm{f}_{L^\infty}+\norm{f_\xi}_{L^2}$. We
define the Banach space $E$ by
\begin{equation*}
  E=V\times H_{0,\infty}(\Real)\times L^2(\Real)\times L^2_{\rm const}(\Real),
\end{equation*}
and 
\begin{equation}
 \bar E=V\times H^1(\Real)\times\Real\times L^2(\Real)\times L^2(\Real)\times \Real.
\end{equation}
Then $E$ is in isometry with $\bar E$ and we define 
\begin{equation}
 \norm{(\zeta,U,h,r)}_{E}=\norm{(\zeta, I_\chi^{-1}(U), h,J^{-1}(r))}_{\bar E}.
\end{equation}
However, all partition functions give rise to equivalent norms (cf. \cite[Sec. 6]{GHR:12}). For convenience, we will
often abuse the notations and denote by the same
$X$ the two elements $(y,\bar U,c,h,\bar r,k)$ and
$(y,U,h,r)$ where, by definition, $U=\bar
U+c\chi\circ y$ and $r=\bar r+ky_\xi$. Furthermore, we will occasionally write $(\zeta,\bar U,c,h,\bar r, k)$ and
$(\zeta,U,h,r)$ for $X$.

The investigation  in \cite[Sec.~3]{GHR:12}  of $g$ defined by \eqref{def:g} showed that $g\in H_{0,\infty}(\Real)$ with $\lim_{x\to\infty}g(x)=1$. Moreover, $g$ is monotonically increasing and $g^\prime\in H^1(\Real)$. 

This yields the following result.
\begin{lemma}
  \label{lem:PQ}
  For any $X=(\zeta,U,h,r)$ in $E$, we define the
  maps $\mathcal{Q}$ and $\mathcal{P}$ as
  $\mathcal{Q}(X)=Q$ and $\mathcal{P}(X)=P$, where
  $P$ and $Q$ are given by \eqref{eq:P} and
  \eqref{eq:Q}, respectively. Then,
  $\mathcal{P}-\frac{1}{2}k^2-U^2$ is a Lipschitz map on bounded
  sets from $E$ to $H^1(\Real)$ and $\mathcal{Q}$
  is a Lipschitz map on bounded sets from $E$ to
  $H^1(\Real)$. Moreover we have
  \begin{equation}\label{eq:Qder}
    Q_\xi= -\frac12 h-(U^2+\frac{1}{2}k^2-P) y_\xi-k\bar r,
  \end{equation}
  \begin{equation}\label{eq:Pder}
    P_\xi=Q(1+\zeta_\xi).
  \end{equation}
\end{lemma}

\begin{proof}
  The structure of the involved functions is
  similar to the ones in \cite{GHR:12}. We
  therefore refer to the proof of \cite[Lemma
  3.1]{GHR:12}.
\end{proof}

\begin{theorem}\label{thm:short}
Given $X_0=(\zeta_0,U_0,h_0,r_0)\in E$, then there exists a time $T$ depending only on $\norm{X_0}_{E}$ such that \eqref{eq:chsyseq} admits a unique solution in $C^1([0,T], E)$ with initial data $X_0$.
\end{theorem}

\begin{proof}
 Solutions of \eqref{eq:chsyseq} can be rewritten as 
\begin{equation*}
 X(t)=X_0+\int_0^t F(X(\tau))d\tau,
\end{equation*}
where $F\colon E\to E$ is defined by the right-hand side of \eqref{eq:chsyseq}. The integrals are defined as Riemann integrals of continuous functions on the Banach space $E$. Using Lemma~\ref{lem:PQ} we can check that  $F(X)$ is a Lipschitz function on bounded sets of $E$. Since $E$ is a Banach space, we use the standard contraction argument to prove the theorem. 
\end{proof}

Differentiating \eqref{eq:chsyseq0} and
\eqref{eq:chsyseq}, we obtain
\begin{subequations}
 \label{eq:govsysder}
\begin{align}
\label{eq:govsysder1}
 y_{\xi,t}& =U_\xi,\\ 
\label{eq:govsysder3}
U_{\xi,t}&=\frac{1}{2}h+(U^2+\frac{1}{2}k^2-P)y_\xi+k\bar r, \\
\label{eq:govsysder4}
h_t&=2(U^2+\frac{1}{2}k^2-P)U_\xi,\\
\label{eq:govsysder5}
r_t&=0,
\end{align}
and
\begin{align}
  \zeta_{\xi,t}&=U_\xi,\\
  \label{eq:govsysder2}
  \bar U_{\xi,t} &= \frac{1}{2}h+(U^2+\frac{1}{2}k^2-P)y_\xi+k\bar r-c\chi''\circ yy_\xi U-c\chi'\circ yU_\xi,\\
  \label{eq:govsysder6}
  \bar r_t&=-kU_\xi.
\end{align}
\end{subequations}
We now turn to the proof of the existence of
global solutions of \eqref{eq:chsyseq}. To that
end we consider initial data that belongs to
$E\cap( [W^{1,\infty}(\Real)]^2\times
[L^\infty(\Real)]^2)$, where
$W^{1,\infty}(\Real)=\{ f\in C_b(\Real) \mid
f_\xi\in L^\infty(\Real)\}$. Given $(\zeta_0, U_0,
h_0, r_0)\in E \cap ([W^{1,\infty}(\Real)]^2\times
[L^\infty(\Real)]^2)$, we consider the short-time
solutions $(\zeta, U, h,r)\in C([0,T],E)$ of
\eqref{eq:chsyseq} given by
Theorem~\ref{thm:short}. Using the fact that
$\mathcal{Q}$ and $\mathcal{P}-\frac{1}{2}k^2-U^2$
are Lipschitz on bounded sets and, since $X\in
C([0,T],E)$, we can prove that
$P-\frac{1}{2}k^2-U^2$ and $Q$ belong to $C([0,T],
H^1(\Real))$. Hence we now consider
$P-\frac{1}{2}k^2-U^2$ and $Q$ as given functions
in $C([0,T], H^1(\Real))$. Then, for any fixed
$\xi\in\Real$, we can solve the system of ordinary
differential equations
\begin{subequations}\label{eq:sysabc}
 \begin{align}
  \frac{d}{dt}\alpha(t,\xi)&=\beta(t,\xi),\\
  \frac{d}{dt}\beta(t,\xi)& = \frac{1}{2}\gamma(t,\xi)+[(U^2+\frac{1}{2}k^2-P)(t,\xi)](1+\alpha(t,\xi))+k\delta(t,\xi),\\
  \frac{d}{dt} \gamma(t,\xi)& = 2(U^2+\frac{1}{2}k^2-P)\beta(t,\xi),\\
  \frac{d}{dt}\delta(t,\xi)&=-k\beta(t,\xi),
\end{align}
\end{subequations}
which is obtained by substituting $\zeta_\xi$, $U_\xi$, $h$, and $\bar r$ in \eqref{eq:govsysder} by the unknowns $\alpha$, $\beta$, $\gamma$, and $\delta$, respectively in $[L^\infty(\Real)]^4$. We have to specify the initial conditions for \eqref{eq:sysabc}. 

Let $\mathcal{A}$ be the following set
\begin{equation*}
 \mathcal{A}=\{ \xi\in\Real \mid  \vert \zeta_{0,\xi}\vert \leq \norm{\zeta_{0,\xi}}_{L^\infty},  \vert U_{0,\xi}(\xi)\vert \leq \norm{U_{0,\xi}}_{L^\infty} , \vert h_0\vert \leq\norm{h_0}_{L^\infty}, \vert r_0\vert\leq\norm{r_0}_{L^\infty}\}.
\end{equation*}

By assumption we have that $\mathcal{A}$ has full measure, that is, $\meas(\mathcal{A}^c)=0$. For $\xi\in \mathcal{A}$ we define $(\alpha(0,\xi), \beta(0,\xi), \gamma(0,\xi),\delta(0,\xi))=(\zeta_{0,\xi}(\xi), U_{0,\xi}(\xi), h_0(\xi), \bar r_0(\xi))$. However, if $\xi\in\mathcal{A}^c$, we set $(\alpha(0,\xi),\beta(0,\xi),\gamma(0,\xi),\delta(0,\xi))=(0,0,0,0)$. 

\begin{lemma}\label{lem:abczUH}
 Given some initial condition $X_0=(\zeta_0,  U_0, h_0,r_0)\in E\cap ([W^{1,\infty}(\Real)]^2\times [L^\infty(\Real)]^2)$, we consider the solution 
$X=(\zeta, U, h,r)\in C^1([0,T],E)$ of \eqref{eq:chsyseq} given by Theorem~\ref{thm:short}. Then $X\in C^1([0,T], E\cap ([W^{1,\infty}(\Real)]^2\times [L^\infty(\Real)]^2))$. The functions $\alpha(t,\xi)$, $\beta(t,\xi)$, $\gamma(t,\xi)$, and $\delta(t,\xi)$ which are obtained by solving \eqref{eq:sysabc} for any fixed given $\xi$ with the initial condition specified above, coincide for almost every $\xi$ and for all times $t$ with $\zeta_\xi$, $U_\xi$, $h$,and $\bar r$, respectively, that is, for all $t\in [0,T]$, we have 
\begin{equation}\label{eq:abczUH}
 (\alpha(t,\xi), \beta(t,\xi),\gamma(t,\xi),\delta(t,\xi))=(\zeta_\xi(t,\xi), U_\xi(t,\xi), h(t,\xi),\bar r(t,\xi))
\end{equation}
   for almost every $\xi\in\Real$.
\end{lemma}

Thus this lemma allows us to pick a special representative for $(\zeta_\xi, U_\xi, h,\bar r)$  given by 
$(\alpha,\beta,\gamma,\delta)$, which is defined for all $\xi\in \Real$ and which, for any given $\xi$, satisfies the ordinary differential equation \eqref{eq:sysabc}. In what follows we will of course identify the two and set $(\zeta_\xi, U_\xi,h,\bar r)$ equal to $(\alpha,\beta,\gamma,\delta)$. 

\begin{proof}
The proof is similar to the one of \cite[Lemma 2.4]{HolRay:07} and hence we refer the interested reader there.
\end{proof}

We define the set $\G$ as follows.

\begin{definition}
\label{def:F}
The set $\G$ is composed of all $(\zeta,U,h,r)\in
E$ such that
\begin{subequations}
\label{eq:lagcoord}
\begin{align}
\label{eq:lagcoord1}
&(\zeta,U,h,r)\in \left[\Winf\right]^2\times [L^\infty(\Real)]^2,\\
\label{eq:lagcoord2}
&y_\xi\geq0, h\geq0, y_\xi+h>0
\text{  almost everywhere},\\
\label{eq:lagcoord3}
&y_\xi h=U_\xi^2+\bar r^2\text{ almost everywhere},
\end{align}
\end{subequations}
where we denote $y(\xi)=\zeta(\xi)+\xi$.
\end{definition}

The set $\G$ is preserved by the flow and for
initial data in $\G$ we obtain global solutions.

\begin{lemma}\label{lem:3.7}
  Given initial data $X_0=(\zeta_0, U_0, h_0,r_0)$ in
  $\mathcal{G}$, let $X(t)=(\zeta(t), U(t), h(t),r(t))$
  be the short-time solution of \eqref{eq:chsyseq}
  in $C^1([0,T],E)$ for some $T>0$ with initial data
  $(\zeta_0, U_0, h_0,r_0)$.  Then
\begin{enumerate}[(i)]
 \item $X(t)$ belongs to $\mathcal{G}$ for all $t\in[0,T]$, 
  \item for almost every $t\in[0,T]$, we have $y_\xi(t,\xi)>0$ for almost every $\xi\in\Real$.
\end{enumerate}
\end{lemma}

\begin{proof}
  The proof follows the same lines as in
  \cite[Lemma 2.7]{HolRay:07}. We prove only
  \eqref{eq:lagcoord3}. From the governing
  equations, for any given $\xi\in\Real$, we get
  \begin{equation*}
    \frac{d}{dt}(y_\xi h)=U_\xi h+2y_\xi(U^2+\frac12k^2-P)U_\xi,
  \end{equation*}
  \begin{equation*}
    \frac{d}{dt}U_\xi^2=2U_\xi(\frac12h+(U^2+\frac12k^2-P)y_\xi+k\bar r)
  \end{equation*}
  and
  \begin{equation*}
    \frac{d}{dt}\bar r^2=-2k\bar rU_\xi.
  \end{equation*}
  Hence, $\frac{d}{dt}(y_\xi h-U_\xi^2-\bar
  r^2)=0$ and \eqref{eq:lagcoord3} is proved.
\end{proof}

We are now ready to prove global existence of solutions to \eqref{eq:chsyseq}.

\begin{theorem}
\label{th:global}
For any $X_0=(y_0,U_0,h_0,r_0)\in\G$, the system \eqref{eq:chsyseq} admits a
unique global solution
$X(t)=(y(t),U(t),h(t),r(t))$ in $C^1(\Real_+,E)$
with initial data $X_0=(y_0,U_0,h_0,r_0)$. We have $X(t)\in\G$ for all times. If we
equip $\G$ with the topology inducted by the
$E$-norm, then the mapping
$S\colon\G\times\Real_+\to\G$ defined as
\begin{equation*}
S_t(\bar X)=X(t)
\end{equation*}
is a continuous semigroup. More precisely, given
$M>0$ and $T>0$, there exists a constant $C_M$
which only depends on $M$ and $T$ such that, for
any two elements $X_\alpha$, $X_\beta\in \G$ such
that $\norm{X_\alpha}_{E}\leq M$,
$\norm{X_\beta}_{E}\leq M$, we have
\begin{equation}\label{est:time}
 \norm{S_tX_\alpha-S_tX_\beta}\leq C_M\norm{X_\alpha-X_\beta}
\end{equation}
for any $t\in[0,T]$.
\end{theorem}

\begin{proof}
  The proof follows closely the one of
  \cite[Theorem 3.7]{GHR:12}, and hence we will
  only prove \eqref{est:time}.  Therefore, one can
  show as in \cite[Theorem 3.7]{GHR:12}, that
\begin{equation}
 \Gamma(t)=\int_\Real \bar U^2y_\xi d\xi+\norm{h}_{L^1},
\end{equation}
satisfies
\begin{equation}
 \frac{d}{dt}\Gamma(t)\leq C(\Gamma(t)+1),
\end{equation}
where $C$ denotes some constant dependent on $c$ (independent of time) and the partition function $\chi$ we choose.
Thus  Gronwall's inequality yields
\begin{equation}\label{est:Gam}
 \Gamma(t)+1\leq (\Gamma(0)+1)e^{Ct},
\end{equation}
where $C$ depends on $c$ and the partition function $\chi$. Moreover, 
\begin{equation}
 \Gamma_0\leq \norm{\bar U_0}_{L^2}^2+\norm{\bar U_0}_{L^\infty}\norm{\bar U_0}_{L^2}\norm{\zeta_{0,\xi}}_{L^2}+\norm{h_0}_{L^1}.
\end{equation}
Hence $\Gamma(t)$ only depends on $t$, the partition function $\chi$ and $\norm{X_0}_{E}$. Thus, using the estimates derived in \cite[Theorem 3.7]{GHR:12}, one obtains that $\norm{S_t(X_\alpha)}_E$ depends only on $\norm{X_0}_{E}$, $t$, and the partition function $\chi$ we choose. Using that $F$ is Lipschitz continuous, we finally end up with \eqref{est:time}. 
\end{proof}

\section{From Eulerian to Lagrangian coordinates and vice versa} \label{sec:eulerlagrange}

\begin{definition} \label{def:D} The set $\D$ is
  composed of all pairs $(u,\rho,\mu)$ such that
  $u\in H_{0,\infty}(\Real)$, $\rho\in L^2_{\rm
    const}(\Real)$ and $\mu$ is a positive finite
  Radon measure whose absolutely continuous part,
  $\muac$, satisfies
\begin{equation}
\label{eq:abspart}
\muac=(u_x^2+\bar\rho^2)\,dx.
\end{equation}
\end{definition}

\begin{definition}\label{def:G}
 We denote by $G$ the subgroup of the group of homeomorphisms from $\Real$ to $\Real$ such that 
\begin{subequations}
\label{eq:Gcond}
 \begin{align}
  \label{eq:Gcond1}
  f-\id \text{ and } f^{-1}-\id &\text{ both belong to } W^{1,\infty}(\Real), \\
  \label{eq:Gcond2}
  f_\xi-1 &\text{ belongs to } L^2(\Real),
 \end{align}
\end{subequations}
where $\id$ denotes the identity function. Given $\kappa>0$, we denote by $G_\kappa$ the subset $G_\kappa$ of $G$ defined by 
\begin{equation}
 G_\kappa=\{ f\in G\mid  \norm{f-\id}_{W^{1,\infty}}+\norm{f^{-1}-\id}_{W^{1,\infty}}\leq\kappa\}. 
\end{equation}
\end{definition}

\begin{lemma}[{\cite[Lemma 3.2]{HolRay:07}}] 
\label{lem:charH}
Let $\kappa\geq0$. If $f$ belongs to $G_\kappa$,
then $1/(1+\kappa)\leq f_\xi\leq 1+\kappa$ almost
everywhere. Conversely, if $f$ is absolutely
continuous, $f-\id\in W^{1,\infty}(\Real)$, $f$ satisfies
\eqref{eq:Gcond2} and there exists $d\geq 1$ such
that $1/d\leq f_\xi\leq d$ almost everywhere, then
$f\in G_\kappa$ for some $\kappa$ depending only
on $d$ and $\norm{f-\id}_{W^{1,\infty}}$.
\end{lemma}
We define the subsets $\mathcal{F}_\kappa$ and $\mathcal{F}$ of $\mathcal{G}$
as follows
\begin{equation*}
\mathcal{F}_\kappa=\{X=(y,U,h,r)\in\mathcal{G}\mid  y+H\in G_\kappa\},
\end{equation*}
and
\begin{equation*}
\mathcal{F}=\{X=(y,U,h,r)\in\mathcal{G}\mid  y+H\in G\},
\end{equation*}
where $H(t,\xi)$ is defined by 
\begin{equation*}
 H(t,\xi)=\int_{-\infty}^\xi h(t,\tau)d\tau,
\end{equation*}
which is finite since, from \eqref{eq:lagcoord3}, we have
$h=U_\xi^2+\bar r^2-\zeta_\xi h$  and therefore $h\in L^1(\Real)$.
For $\kappa=0$, we have $G_0=\{\id\}$. As we shall see,
the space $\mathcal{F}_0$ will play a special role. These
sets are relevant only because they are preserved
by the governing equation \eqref{eq:chsyseq} as the
next lemma shows. In particular, while the mapping
$\xi\mapsto y(t,\xi)$ may not be a diffeomorphism
for some time $t$, the mapping $\xi\mapsto
y(t,\xi)+H(t,\xi)$ remains a diffeomorphism for
all times $t$.

\begin{lemma}\label{lem:Fpres}
 The space $\mathcal{G}$ is preserved by the governing equations \eqref{eq:chsyseq}. More precisely, given $\kappa$, $T\geq 0$, and $X_0\in \mathcal{G}_\kappa$, we have 
\begin{equation*}
 S_t(X_0)\in \mathcal{G}_{\kappa'}, 
\end{equation*}
for all $t\in[0,T]$ where $\kappa'$ only depends on $T$, $\kappa$, and $\norm{X_0}_E$.
\end{lemma}

\begin{proof}  
See \cite[Lemma 3.3]{HolRay:07}.
\end{proof}

For the sake of simplicity, for any $X=(y, U, h,r)\in \mathcal{F}$ and any function
$f\in\Gr$, we denote $(y\circ f, U\circ f, h\circ ff_\xi, r\circ f f_\xi)$ by $X\circ f$.

\begin{proposition} The map 
from $\Gr\times\F$ to $\F$ given by $(f,X)\mapsto
X\circ f$ defines an action of the group $\Gr$ on
$\F$.
\end{proposition}

\begin{proof} 
See \cite[Proposition 3.4]{HolRay:07}. 
\end{proof}

Since $\Gr$ is acting on $\F$, we can consider the
quotient space $\quot$ of $\F$ with respect to the
action of the group $G$. The equivalence relation
on $\F$ is defined as follows: For any
$X,X'\in\F$, we say that $X$ and $X'$ are equivalent if there
exists $f\in\Gr$ such that $X'=X\circ f$. We
denote by $\Pi(X)=[X]$ the projection of $\F$ into the
quotient space $\quot$, and introduce the mapping
$\Gamma\colon\F\rightarrow\F_0$ given by
\begin{equation*}
\Gamma(X)=X\circ (y+H)\inv
\end{equation*}
for any $X=(y,U,h,r)\in\F$. We have $\Gamma(X)=X$ when
$X\in\F_0$. It is not hard to prove that $\Gamma$ is
invariant under the $\Gr$ action, that is,
$\Gamma(X\circ f)=\Gamma(X)$ for any $X\in\F$ and
$f\in\Gr$. Hence, there corresponds to $\Gamma$ a
mapping $\tilde\Gamma$ from the quotient space $\quot$ to
$\F_0$ given by $\tilde\Gamma([X])=\Gamma(X)$ where
$[X]\in\quot$ denotes the equivalence class of
$X\in\F$. For any $X\in\F_0$, we have
$\tilde\Gamma\circ\Pi(X)=\Gamma(X)=X$. Hence, $\tilde\Gamma\circ
\Pi|_{\F_0}=\id|_{\F_0}$. Any topology defined on
$\F_0$ is naturally transported into $\quot$ by
this isomorphism. We equip $\F_0$ with the metric
induced by the $E$-norm, i.e.,
$d_{\F_0}(X,X')=\norm{X-X'}_E$ for all
$X,X'\in\F_0$. Since $\F_0$ is closed in $E$, this
metric is complete. We define the metric on
$\quot$ as
\begin{equation*}
d_\quot([X],[X'])=\norm{\Gamma(X)-\Gamma(X')}_E,
\end{equation*}
for any $[X],[X']\in\quot$. Then, $\quot$ is
isometrically isomorphic with $\F_0$ and the
metric $d_\quot$ is complete.

\begin{lemma}
\label{lem:picont} Given $\alpha\geq 0$.
The restriction of $\Gamma$ to $\F_\alpha$ is a
continuous mapping from $\F_\alpha$ to $\F_0$.
\end{lemma}
\begin{proof}   
See \cite[Lemma 3.5]{HolRay:07}.
\end{proof}

\begin{remark}
The mapping $\Gamma$ is not continuous from $\F$ to
$\F_0$. The spaces $\F_\alpha$ were precisely
introduced in order to make the mapping $\Gamma$
continuous.
\end{remark}

We denote by $S\colon\F\times [0,\infty)\rightarrow \F$
the continuous semigroup which to any initial
data $X_0\in \F$ associates the solution $X(t)$
of the system of differential equations
\eqref{eq:chsyseq} at time $t$. As indicated earlier,
the two-component Camassa--Holm system is invariant with
respect to relabeling.  More precisely, using our
terminology, we have the following result.

\begin{theorem} 
\label{th:sgS} 
For any $t>0$, the mapping $S_t\colon\F\rightarrow\F$
is $\Gr$-equivariant, that is,
\begin{equation}
\label{eq:Hequivar}
S_t(X\circ f)=S_t(X)\circ f
\end{equation}
for any $X\in\F$ and $f\in\Gr$. Hence, the mapping
$\tilde S_t$ from $\quot$ to $\quot$ given by
\begin{equation*}
\tilde S_t([X])=[S_tX]
\end{equation*}
is well-defined. It generates a continuous
semigroup.
\end{theorem}

\begin{proof}  
See \cite[Theorem 3.7]{HolRay:07}.
\end{proof}

We have the
following diagram:
\begin{equation}
\label{eq:diag}
\xymatrix{
\F_0\ar[r]^{\Pi}&\quot\\
\F_\alpha\ar[u]^{\Gamma}&\\
\F_0\ar[u]^{S_t}\ar[r]^\Pi&\quot\ar[uu]_{\tilde S_t}
}
\end{equation}

\subsection{Mappings between the two coordinate systems}

Our next task is to derive the correspondence
between Eulerian coordinates (functions in $\D$)
and Lagrangian coordinates (functions in
$\quot$). The set $\D$ however allows the energy
density to have a singular part and a positive
amount of energy can concentrate on a set of
Lebesgue measure zero.

We define the mapping $L$ from $\D$ to $\F_0$ which to
any initial data in $\D$ associates an initial
data for the equivalent system in $\F_0$.

\begin{theorem} 
\label{th:Ldef}
For any $(u,\rho,\mu)$ in $\D$, let
\begin{subequations}
\label{eq:Ldef}
\begin{align}
\label{eq:Ldef1}
y(\xi)&=\sup\left\{y\ |\ \mu((-\infty,y))+y<\xi\right\},\\
\label{eq:Ldef2}
h(\xi)&=1-y_\xi(\xi),\\
\label{eq:Ldef3}
U(\xi)&=u\circ{y(\xi)},\\
\label{eq:Ldef4}
r(\xi)&=\rho\circ{y(\xi)}y_\xi(\xi).
\end{align}
\end{subequations}
Then $(y,U,h,r)\in\F_0$. We denote by $L\colon \D\rightarrow \F$ the mapping which to any element $(u,\rho,\mu)\in\D$ associates $X=(y,U,h,r)\in \F$ given by \eqref{eq:Ldef}.  
\end{theorem}

\begin{proof}
  The proof follows the same lines as the one of
  \cite[Theorem 3.8]{HolRay:07}, which we do not
  repeat here.  We give some more details about
  the variable $\rho$ and its Lagrangian
  counterpart $r$, which are specific to the
  Camassa-Holm system. As in \cite[Theorem
  3.8]{HolRay:07}, we obtain that
  \begin{equation}
    \label{eq:fondideul}
    (u_x^2\circ y+\bar\rho^2\circ y+1)y_\xi=1
  \end{equation}
  for allmost every $\xi$ such that
  $y_\xi\neq0$. From \eqref{eq:fondideul}, it
  follows that $0\leq y_\xi\leq 1$ and
  $hy_\xi=(u_\xi\circ y)^2y_\xi^2+(\bar\rho\circ
  y)^2 y_\xi^2$, which implies
  \eqref{eq:lagcoord3}. It remains to prove that
  $\bar r\in L^2(\Real)\cap L^\infty(\Real)$. We
  have
  \begin{equation*}
    \int_\Real\bar r(\xi)^2\,d\xi=\int_\Real\bar\rho(y(\xi))^2y_\xi^2\,d\xi\leq\int_\Real\bar\rho(y(\xi))^2y_\xi\,d\xi=\int_\Real\bar\rho(x)^2\,dx<\infty
  \end{equation*}
  and
  \begin{equation*}
    \bar r^2=\bar\rho^2\circ
    y\,y_\xi^2\leq\bar\rho^2\circ
    y\,y_\xi\leq 1.
  \end{equation*}
\end{proof}

Reversely, to any element in $\F$ there
corresponds a unique element in $\D$ which is
given by the mapping $M$ defined below.

\begin{theorem}
\label{th:umudef} 
Given any element $X=(y,U,h,r)\in\F$. Then, the
measure $y_{\#}(\bar r(\xi)\,d\xi)$ is absolutely
continuous, and we define $(u,\rho,\mu)$ as follows
\begin{subequations}
\label{eq:umudef}
\begin{align}
\label{eq:umudef1}
&u(x)=U(\xi)\text{ for any }\xi\text{ such that  }  x=y(\xi),\\
\label{eq:umudef2}
&\mu=y_\#(h(\xi)\,d\xi),\\
\label{eq:umudef3}
&\bar \rho(x)\,dx=y_\#(\bar r(\xi)\,d\xi),\\
\label{eq:umudef4}
&\rho(x)=k+\bar\rho(x).
\end{align}
\end{subequations}
We have that $(u,\rho,\mu)$ belongs to $\D$. We
denote by $M\colon \F\rightarrow\D$ the mapping
which to any $X$ in $\F$ associates the element $(u, \rho,\mu)\in \D$ as given
by \eqref{eq:umudef}. In particular, the mapping $M$ is
invariant under relabeling.
\end{theorem}

\begin{proof}
  Most of the proof follows closely \cite[Theorem
  3.11]{HolRay:07}. That the definition of $u$ is
  well-posed follows exactly as in \cite[Theorem
  3.11]{HolRay:07}. Next we prove that
  $y_{\#}(\bar r(\xi)\,d\xi)$ is absolutely
  continuous and satisfies \eqref{eq:umudef3} for
  $\bar \rho\in L^2(\Real)$. In addition, we will
  see that $\bar\rho\circ yy_\xi=\bar r$ whenever
  $y_\xi\not=0$. Let $\lambda=y_\#(\bar
  r(\xi)\,d\xi)$. For any continuous function with
  compact support $\phi$, we have, using the
  change of variables $x=y(\xi)$,
  \begin{equation*}
    \int_\Real \phi(x)d\lambda(x)=\int_{\Real}
    \bar r\phi\circ y \,d\xi=\int_{\{\xi\in\Real\
      |\ y_\xi(\xi)\neq 0\}} \frac{\bar
      r}{\sqrt{y_\xi}}\phi\circ y
    \sqrt{y_\xi}\,d\xi,
  \end{equation*}
  because $\bar r(\xi)=0$ for almost every $\xi$ such
  that $y_\xi=0$. Hence,
  \begin{align}
    \notag
    \int_\Real \phi(x)d\lambda(x)& \leq
    \Big(\int_{\{\xi\in\Real\
      |\ y_\xi(\xi)\neq 0\}} \frac{\bar r^2}{y_\xi}\,d\xi\Big)^{1/2}\Big(\int_\Real \phi^2 \dx\Big)^{1/2}\\
    \label{eq:L2bd}
    & \leq \Big(\int_\Real h
    \,d\xi\Big)^{1/2}\norm{\phi}_{L^2}\leq
    \norm{h}_{L^1}^{1/2}\norm{\phi}_{L^2},
  \end{align}
  where we used that $\bar r^2\leq h y_\xi$, see
  \eqref{eq:lagcoord3}. It follows that $\lambda$
  is absolutely continuous, we write
  $\lambda=\bar\rho dx$, and \eqref{eq:L2bd} also
  implies that $\bar\rho\in L^2(\Real)$ with
  $\norm{\bar\rho}_{L^2}\leq\norm{h}_{L^1}^{\frac{1}2}$.
  By the definition of $\bar\rho$, we have that,
  for any set $B$,
  \begin{equation}
    \label{eq:barrhor}
    \int_{B}\bar\rho(x)\,dx=\int_{y^{-1}(B)}\bar\rho\circ yy_\xi\,d\xi=\int_{y^{-1}(B)}\bar r\,d\xi
  \end{equation}
  Define
  \begin{align*}
    Z=\{ \xi\in\Real&\mid  y\text{ is differentiable at }\xi \text{ and }y_\xi(\xi)=0\\
    &\quad \text{ or }y \text{ is not
      differentiable at }\xi\},
  \end{align*}
  For any $\tilde B\subset Z^c$, we have
  $y^{-1}(y(\tilde B))=\tilde B$. Indeed, assume
  the opposite. Then, there exists $\xi_0\in\tilde
  B$ and $\xi_1\in\tilde B^c$ such that
  $y(\xi_0)=y(\xi_1)$. Since $y$ is increasing, it
  implies that $y_\xi(\xi_0)=0$, which contradicts
  the fact that $\xi_0\in Z^c$. Thus,
  \eqref{eq:barrhor} gives
  \begin{equation*}
    \int_{\tilde B}\bar\rho\circ yy_\xi\,d\xi=\int_{\tilde B}\bar r\,d\xi,
  \end{equation*}
  and it follows that 
  \begin{equation}
    \label{eq:identrhor}
    \bar\rho\circ yy_\xi=\bar r
  \end{equation}
  on $Z^c$. Finally, it is left to show that
  $\mu_{ac}=u_x^2+\bar\rho^2$. We can use
  \eqref{eq:identrhor} and the proof follows the
  same line as in \cite[Theorem 3.11]{HolRay:07}.
\end{proof}

Finally, in order to show that the equivalence
classes in Lagrangian coordinates are in bijection
with the set of Eulerian coordinates, it is left
to prove the following theorem.

\begin{theorem}\label{th:LMinv}The mappings $M$ and
  $L$ are invertible. We have
\begin{equation*}
L\circ M=\id_\quot\text{ and }M\circ L=\id_\D.
\end{equation*}
\end{theorem}

\begin{proof}
Given $[X]$ in $\quot$, we choose
$X=(y,U,h,r)=\tilde \Gamma([X])$ as a representative of
$[X]$ and consider $(u,\rho,\mu)$ given by
\eqref{eq:umudef} for this particular $X$. Note
that, from the definition of $\tilde \Gamma$, we have
$X\in\F_0$. Let $\tilde X=(\tilde y,\tilde U,\tilde h,\tilde r)$ be
the representative of $L(u,\rho,\mu)$ in $\F_0$ given
by the formulas \eqref{eq:Ldef}. We claim that
$(\tilde y,\tilde U,\tilde h,\tilde r)=(y,U,h,r)$ and therefore
$L\circ M=\id_\quot$. Let
\begin{equation}
\label{eq:gdef}
g(x)=\sup\{\xi\in\Real\ |\ y(\xi)<x\}.
\end{equation}
It is not hard to prove, using the fact that $y$
is increasing and continuous, that 
\begin{equation}
\label{eq:ygxeqx}
y(g(x))=x
\end{equation}
and $y\inv((-\infty,x))=(-\infty,g(x))$. For any
$x\in\Real$, we have, by \eqref{eq:umudef2}, that
\begin{equation*}
\mu((-\infty,x))=\int_{y\inv((-\infty,x))}h\,d\xi
=\int_{-\infty}^{g(x)}h\,d\xi
=H(g(x))
\end{equation*}
because $H(-\infty)=0$. Since $X\in\F_0$,
$y+H=\id$ and we get
\begin{equation}
\label{eq:mupxeqg}
\mu((-\infty,x))+x=g(x).
\end{equation}
From the definition of $\tilde y$, we then obtain
that
\begin{equation}
\label{eq:ybardef}
\tilde y(\xi)=\sup\{x\in\Real\ |\ g(x)<\xi\}.
\end{equation}
For any given $\xi\in\Real$, let us consider an
increasing sequence $x_i$ tending to $\tilde
y(\xi)$ such that $g(x_i)<\xi$; such sequence
exists by \eqref{eq:ybardef}. Since $y$ is
increasing and using \eqref{eq:ygxeqx}, it follows
that $x_i\leq y(\xi)$. Letting $i$ tend to
$\infty$, we obtain $\tilde y(\xi)\leq
y(\xi)$. Assume that $\tilde y(\xi)<y(\xi)$. Then,
there exists $x$ such that $\tilde
y(\xi)<x<y(\xi)$ and equation \eqref{eq:ybardef}
then implies that $g(x)\geq\xi$. On the other
hand, $x=y(g(x))<y(\xi)$ implies $g(x)<\xi$
because $y$ is increasing, which gives us a
contradiction. Hence, we have $\tilde y=y$. It
follows directly from the definitions, since
$y_\xi+h=1$, that $\tilde h=h$, $\tilde U=U$,
$\vert \bar{ \tilde {r}}\vert =\vert \bar r\vert$,
and $\tilde d= d$. If $y_\xi(\xi)=0$, then $\bar
{\tilde {r}}(\xi)=0=\bar r(\xi)$. We use
\eqref{eq:identrhor} and get $\bar {\tilde
  {r}}(\xi)=\bar\rho(\tilde y(\xi))\tilde
y_\xi=\bar\rho( y(\xi)) y_\xi=\bar r$ if $y_\xi(\xi)\not=0$. Thus we
have proved that $L\circ M=\id_\quot$.

\medskip
We now turn to the proof that $M\circ L=\id_\D$. 
Given $(u,\rho,\mu)$ in $\D$, we denote by $(y,U,h,r)$
the representative of $L(u,\rho,\mu)$ in $\F_0$ given
by \eqref{eq:Ldef}. Then, let $(\tilde
u,\tilde \rho,\tilde\mu)=M\circ L(u,\rho,\mu)$.  We claim that
$(\tilde u,\tilde\rho,\tilde \mu)=(u,\rho,\mu)$. Let $g$ be the
function defined as before by \eqref{eq:gdef}. The
same computation that leads to \eqref{eq:mupxeqg}
now gives
\begin{equation}
\label{eq:mubarg}
\tilde \mu((-\infty,x))+x=g(x).
\end{equation}
Given $\xi\in\Real$, we consider an increasing
sequence $x_i$ which converges to $y(\xi)$ and
such that $\mu((-\infty,x_i))+x_i<\xi$. The
existence of such a sequence is guaranteed by
\eqref{eq:Ldef1}. Passing to the limit and since
$F(x)=\mu((-\infty,x))$ is lower semi-continuous,
we obtain
$\mu((-\infty,y(\xi)))+y(\xi)\leq\xi$. We take
$\xi=g(x)$ and get
\begin{equation}
\label{eq:musg}
\mu((-\infty,x))+x\leq g(x).
\end{equation}
From the definition of $g$, there exists an
increasing sequence $\xi_i$ which converges to
$g(x)$ such that $y(\xi_i)<x$. The definition
\eqref{eq:Ldef1} of $y$ tells us that
$\mu((-\infty,x))+x\geq\xi_i$. Letting $i$ tend to
infinity, we obtain $\mu((-\infty,x))+x\geq g(x)$
which, together with \eqref{eq:musg}, yields
\begin{equation}
\label{eq:mueqg}
\mu((-\infty,x))+x=g(x).
\end{equation}
Comparing \eqref{eq:mueqg} and \eqref{eq:mubarg}
we get that $\mu=\tilde\mu$.  It is clear from the
definitions that $\tilde u=u$. Using
\eqref{eq:identrhor}, we get
$\tilde\rho(y)y_\xi=r=\rho(y)y_\xi$ and therefore
$\tilde\rho=\rho$. Hence, $(\tilde
u,\tilde\rho,\tilde\mu)=(u,\rho,\mu)$ and $M\circ
L=\id_{\D}$.
\end{proof}

\section{Continuous semigroup of solutions}

We define $T_t$ as
\begin{equation*}
  T_t=M\circ S_t\circ L.
\end{equation*}

The metric $d_\D$ is defined as 
\begin{equation*}
  d_D((u_1,\rho_1,\mu_1),(u_2,\rho_2,\mu_2))=d_{\F_0}(L(u_1,\rho_1,\mu_1),L(u_2,\rho_2,\mu_2)).
\end{equation*}

\begin{definition}
  \label{eq:defweakconssol}
  Assume that $u\colon[0,\infty)\times\Real \to \Real$  and $\rho\colon [0,\infty) \times\Real \to\Real$ satisfy \\
  (i) $u\in L^\infty_{\rm{loc}}([0,\infty), H_\infty(\Real))$ and $\rho\in L^\infty_{\rm{loc}}([0,\infty), L^2_{\rm const}(\Real))$, \\
  (ii) the equations
  \begin{multline}\label{eq:weak1}
    \iint_{[0,\infty)\times\Real}\Big[
    -u(t,x)\phi_t(t,x)
    +\big(u(t,x)u_x(t,x)+P_x(t,x)\big)\phi(t,x)\Big]dxdt\\
    =\int_\Real u(0,x)\phi(0,x)dx,
  \end{multline}
  \begin{equation}\label{eq:weak2}
    \iint_{[0,\infty)\times\Real} \Big[(P(t,x)-u^2(t,x)-\frac{1}{2}u_x^2(t,x)-\frac{1}{2}\rho^2(t,x))\phi(t,x)+P_x(t,x)\phi_x(t,x)\Big]dxdt=0,
  \end{equation}
  and
  \begin{equation}
   \label{eq:weak3}
   \iint_{[0,\infty)\times\Real}\Big[ -\rho(t,x)\phi_t(t,x)-u(t,x)\rho(t,x)\phi_x(t,x)\Big]dxdt=\int_\Real \rho(0,x)\phi(0,x)dx,
  \end{equation}
  hold for all $\phi\in
  C^\infty_0([0,\infty)\times\Real)$. Then we say that
  $u$ is a weak global solution of the two-component
  Camassa--Holm system. If $u$ in addition satisfies
  \begin{equation*}
    (u^2+u_x^2+\rho^2)_t+(u(u^2+u_x^2+\rho^2))_x-(u^3-2Pu)_x=0
  \end{equation*}
  in the sense that
  \begin{align}
    \iint_{(0,\infty)\times\Real}&\Big[
    (u^2(t,x)+u_x^2(t,x)+\rho^2(t,x))\phi_t(t,x)\notag\\
    &+(u(t,x)(u^2(t,x)+u_x^2(t,x)+\rho^2(t,x)))\phi_x(t,x)\label{eq:weak4}\\
    &\qquad\qquad-(u^3(t,x)-2P(t,x)u(t,x))\phi_x(t,x)\Big]dxdt \notag
    =0,
  \end{align}
  for any $\phi\in
  C_0^{\infty}((0,\infty)\times\Real)$, we say that $u$
  is a weak global conservative solution of the two-component
  Camassa--Holm system.
\end{definition}

\begin{theorem} \label{th:mainX}
  The mapping $T_t$ is a continuous
  semigroup of solutions with respect to the metric
  $d_\D$.  Given any initial data
  $(u_0,\rho_0,\mu_0)\in\D$, let
  $(u(t,\dott),\rho(t,\dott)),\mu(t,\dott)))=T_t(u_0,\rho_0,\mu_0)$. Then
  $(u,\rho)$ is a weak solution to
  \eqref{eq:chsys}, and $\mu$ is a weak solution to
  \begin{equation*}
    (u^2+\mu+\rho^2-\bar \rho^2)_t+(u(u^2+\mu+\rho^2-\bar \rho^2))_x=(u^3-2Pu)_x.
  \end{equation*}
  For almost every time, $\mu$ is absolutely
  continuous, and in that case
  $\mu=(u_x^2+\bar \rho^2)\,dx$.
\end{theorem}

\begin{proof}
The proof follows closely the one for the Camassa--Holm equation with nonvanishing asymptotics, and we therefore refer to \cite[Theorem 5.2]{GHR:12}.
 
\end{proof}

\section{Regularity and uniqueness results}

Given $(u,\rho,\mu)\in\D$, $p\in\mathbb{N}$ and an
open set $I$, we say that $(u,\rho,\mu)$ is
$p$-regular on an open set $I$ if
\begin{equation*}
  u\in W^{p,\infty}(I),\ \rho\in W^{p-1,\infty}(I)\ \text{ and } \muac=\mu\text{ on }I.
\end{equation*}
By notation, we set
$W^{0,\infty}(I)=L^\infty(I)$. The variable $\rho$
has a regularizing effect, as the following
theorem shows. Even if the 2CH system has an
infinite speed of propagation, see
\cite{henry:09}, we obtain the hyperbolic feature
that discontinuities travel at finite speed. The
regularity is preserved in intervals defined by the
characteristics. 
\begin{theorem}
  \label{th:presreg}
  We consider the initial data
  $(u_0,\rho_0,\mu_0)$. Assume that
  $(u_0,\rho_0,\mu_0)$ is p-regular on a given
  interval $(x_0,x_1)$ and
  \begin{equation}
    \label{eq:rhopos}
    \rho_0(x)^2\geq c>0
  \end{equation}
  for $x\in (x_0,x_1)$. Then, for any
  $t\in\Real_+$, $(u,\rho,\mu)(t,\cdot)$ is
  p-regular on the interval $(y(t,\xi_0),
  y(t,\xi_1))$, where $\xi_0$ and $\xi_1$ satisfy
  $y(0,\xi_0)=x_0$ and $y(0,\xi_1)=x_1$ and are
  defined as
  \begin{equation*}
    \xi_0=\sup\{\xi\in\Real\ |\ y(0,\xi)\leq x_0\} \text{ and }
    \xi_1=\inf\{\xi\in\Real\ |\ y(0,\xi)\geq x_1\}.
  \end{equation*}
\end{theorem}
\begin{proof}
  We consider first the case $p=1$. Let
  $(y_0,U_0,h_0,r_0)=L(u_0,\mu_0,\rho_0)$. Since
  $y_0$ is surjective and continuous, we have
  $y_0(\xi_0)=x_0$ and $y_0(\xi_1)=x_1$. We denote
  $I=(x_0,x_1)$ and $J=(\xi_0,\xi_1)$. Since $\mu$
  is absolutely continuous on $I$, we get from
  \eqref{eq:Ldef1} that, for any $\xi\in J$,
  \begin{equation*}
    \mu((-\infty,y_0(\xi_0)])+\muac((y_0(\xi_0),y_0(\xi)))+y_0(\xi)=\xi
  \end{equation*}
  so that
  \begin{equation*}
    \mu((-\infty,y_0(\xi_0)])+\int_{y_0(\xi_0)}^{y_0(\xi)}(u_{0x}^2+\bar\rho_0^2)\,dx+y_0(\xi)=\xi.
  \end{equation*}
  We differentiate this relation with respect to
  $\xi$ and obtain that
  \begin{equation}
    \label{eq:1pu0}
    y_{0\xi}(\xi)=\frac{1}{(1+u_{0x}^2+\bar\rho_0^2)\circ y_0(\xi)}.
  \end{equation}
  Hence,
  \begin{equation}
    \label{eq:lowbdy0xi}
    y_{0\xi}(\xi)\geq \frac{1}{1+C}
  \end{equation}
  for any $\xi\in J$ and for a constant $C$ which
  depends only on $\norm{u_0}_{W^{1,\infty}}(I)$
  and $\norm{\rho_0}_{L^\infty(I)}$, both
  quantities being bounded when
  $(u_0,\rho_0,\mu_0)$ is $p$-regular in $I$. By
  the definition of $r_0$, we have
  $\rho_0(y_0(\xi))y_{0\xi}(\xi)=r_0(\xi)$ and,
  therefore, it follows from \eqref{eq:lowbdy0xi}
  and assumption \eqref{eq:rhopos} that, for any
  $\xi\in J$, $r_0(\xi)\geq c_1$ for some constant
  $c_1>0$. By \eqref{eq:govsysder5} and
  \eqref{eq:lagcoord3}, we have
  \begin{equation}
    \label{eq:bdbelr0}
    c_1^2\leq r_0^2(\xi)=r^2(t,\xi)=(\bar r+ky_\xi)^2=\bar r^2+y_\xi(2\bar r+y_\xi)\leq y_\xi(h+2\bar r+y_\xi)
  \end{equation}
  for any $t\in[0,T]$ and $\xi\in J$. In Lemma
  \ref{lem:3.7}, it is shown that
  \begin{equation*}
    \norm{\zeta(t,\cdot)}_{W^{1,\infty}}+\norm{U(t,\cdot)}_{W^{1,\infty}}+\norm{h(t,\cdot)}_{L^\infty}+\norm{r(t,\cdot)}_{L^\infty}\leq C_1
  \end{equation*}
  for some constant $C_1$ which depends on the
  initial data. Then, \eqref{eq:bdbelr0} yields
  $y_\xi(t,\xi)\geq c_2$ for some constant
  $c_2>0$. It follows that for each $t\in[0,T]$,
  the mapping
  $y(t,\cdot)\colon(\xi_0,\xi_1)\mapsto(y(t,\xi_0),y(t,\xi_1))$
  is a Lipschitz homeomorphism and its inverse
  $y^{-1}$ is also Lipschitz. We denote the
  interval $(y(t,\xi_0),y(t,\xi_1))$ by $I^t$. By
  the definitions of Theorem \ref{th:umudef}, it
  follows that
  \begin{equation}
    \label{eq:defurho}
    u(t,x)=U(t,y^{-1}(t,x))\ \text{ and }\ \rho(t,x)=\frac{r(t,y^{-1}(t,x))}{y_\xi(t,y^{-1}(t,x))}
  \end{equation}
  for $x\in I^t$. Hence, $u(t,\cdot)\in
  W^{1,\infty}(I^t)$ and $\rho(t,\cdot)\in
  L^\infty(I^t)$. Since $\mu=y_\#(h(\xi)\,d\xi)$
  and $y^{-1}$ is Lipschitz on $I^t$, we have
  \begin{equation*}
    \mu(t,A)=\int_{y^{-1}(A)}h(t,\xi)\,d\xi=\int_A\frac{h\circ y^{-1}}{y_\xi\circ y^{-1}}\,dx=\int_A(u_x^2(t,x)+\bar\rho(t,x)^2)dx
  \end{equation*}
  for any subset $A$ of $I^t$. Hence, $\mu=\muac$
  in $I^t$. Let us now consider the case $p>1$. It
  follows from \eqref{eq:1pu0} that $y_0\in
  W^{p,\infty}(J)$. Since $U_0(\xi)=u_0\circ
  y_0(\xi)$ on $J$, we get $U_0\in
  W^{p,\infty}(J)$. Similarly, \eqref{eq:Ldef2}
  and \eqref{eq:Ldef4} yield $h_0\in
  W^{p-1,\infty}(J)$ and $r_0\in
  W^{p-1,\infty}(J)$. The key point is that, due
  to the quasilinear structure of the equivalent
  system \eqref{eq:govsysder}, this regularity is
  preserved by the flow. By differentiating $Q$ given by \eqref{eq:Q},
  we observe that $\partial_\xi^pQ$ is quasilinear
  in $\partial_\xi^py$, $\partial_\xi^p U$,
  $\partial_\xi^{p-1}h$ and
  $\partial_\xi^{p-1}\bar r$, that is, it can be
  written as
  \begin{equation}
    \label{eq:derpQ}
    \partial_\xi^pQ_\xi=a_1\partial_\xi^py_\xi+a_2\partial_\xi^pU_\xi+a_3\partial_\xi^{p-1}h+a_4\partial_\xi^{p-1}\bar r+a_5
  \end{equation}
  where $\{a_i\}_{i=1}^{5}$ are functions that
  are bounded in $L^\infty(J)$ by a constant
  depending only on
  $\norm{y-\id}_{W^{p-1,\infty}(J)}$,
  $\norm{U}_{W^{p-1,\infty}(J)}$,
  $\norm{h}_{W^{p-2,\infty}(J)}$, $\norm{\bar
    r}_{W^{p-2,\infty}(J)}$,
  $\norm{P}_{L^\infty(J)}$ and
  $\norm{Q}_{L^\infty(J)}$. The same property
  holds for $P$. One checks directly that the
  system \eqref{eq:govsysder} inherits the same
  quasilinearity property, that is, the $p$th
  derivative of each term on the right-hand side
  can be written as a linear combination of
  $\partial_\xi^py$, $\partial_\xi^p U$,
  $\partial_\xi^{p-1}h$ and
  $\partial_\xi^{p-1}\bar r$ as in
  \eqref{eq:derpQ}. Then, we use Gronwall's lemma
  and an induction argument on $j=1,\dots,p$ to
  prove, that if
  \begin{equation*}
    (\zeta(t,\cdot),U(t,\cdot),h(t,\cdot),r(t,\cdot))\in W^{j,\infty}(J)\times
    W^{j,\infty}(J)\times W^{j-1,\infty}(J)\times
    W^{j-1,\infty}(J)
  \end{equation*}
  holds for $t=0$, then it remains true for
  $t\in[0,T]$, for $j=1,\dots,p$. Using
  \eqref{eq:defurho}, we conclude that $u\in
  W^{p,\infty}(I^t)$ and $\rho\in
  W^{p-1,\infty}(I^t)$.
\end{proof}

\begin{corollary}\label{cor:presreg3}
  If the initial data $(u_0,\rho_0,\mu_0)\in\D$
  satisfies $u_0,\rho_0\in C^{\infty}(\Real)$,
  $\mu_0$ is absolutely continuous and
  $\rho_0^2(x)\geq d>0$ for all $x\in\Real$, then
  $u,\rho\in C^\infty(\Real\times\Real)$
  is the unique classical solution to
  \eqref{eq:rewchsys}.
\end{corollary}

These classical solutions can be used to obtain
the global conservative solution of the
Camassa--Holm equation \eqref{eq:ch}. We consider
only initial data for which $\mu_0$ is absolutely
continuous.

\begin{theorem}
  \label{th:approxCH}
  Let $u_0\in H_\infty(\Real)$. We consider the
  approximating sequence of initial data
  $(u_0^n,\rho_0^n,\mu_{0}^n)\in\D$ given by
  $u_0^n\in C^\infty(\Real)$ with
  $\lim_{n\to\infty}u_0^n=u_0$ in
  $H_\infty(\Real)$, $\rho_0^n\in C^\infty(\Real)$
  with $\lim_{n\to\infty}\rho_0^n=0$ in
  $L^2_{\rm const}(\Real)$, $(\rho_0^n)^2\geq d_n$ for
  some constant $d_n>0$ and for all $n$ and
  $\mu_{0}^n=((u_{0,x}^{n})^2+(\bar\rho_0^n)^2)\,dx$. We
  denote by $(u^n,\rho^n)$ the unique classical
  solution to \eqref{eq:rewchsys} in
  $C^\infty(\Real_+\times\Real)\times
  C^\infty(\Real_+\times\Real)$ which corresponds
  to this initial data.  Then for every
  $t\in\Real_+$, the sequence $u^n(t,\cdot)$
  converges to $u(t,\cdot)$ in
  $L^{\infty}(\Real)$, where $u$ is the
  conservative solution of the Camassa--Holm
  equation \eqref{eq:ch} with initial data $u_0\in
  H_\infty(\Real)$.
\end{theorem}

\begin{proof}
  The proof relies on the stability of the
  semigroup of solutions and on the Lemmas
  \ref{thm:top1} and \ref{thm:top2} below, which
  compare the topology of $\D$ with standard
  topologies. From Lemma~\ref{thm:top1}, it
  follows that $(u_0^n,\rho_0^n,\mu_0^n)$ converges
  to $(u_0,0,\mu_0)$ in $\D$. By the stability of
  the semigroup with respect to $\D$, we have
  that, for any $t\in\Real_+$,
  $(u^n(t),\rho^n(t),\mu^n(t))$ converges to
  $(u(t),0,\mu(t))$ in $\D$. Then, Lemma
  \ref{thm:top2} gives that $u^n(t,\cdot)$
  converges to $u(t,\cdot)$ in $L^\infty(\Real).$
\end{proof}

In order to complete the proof of the previous and
main theorem, we show the following lemmas.
\begin{lemma}\label{thm:top1}
 The mapping
\begin{equation}
 (u,\rho)\mapsto (u,\rho, (u_x^2+ \bar\rho^2)dx)
\end{equation}
is continuous from $H_{0,\infty}(\Real)\times
L^2_{\rm const}(\Real)$ into $\D$. In other words,
given a sequence $(u_n,\rho_n)$ in
$H_{0,\infty}(\Real)\times L^2_{\rm const}(\Real)$
which converges to $(u,\rho)$ in
$H_{0,\infty}(\Real)\times L^2_{\rm const}(\Real)$,
then $(u_n,\rho_n,(u_{n,x}^2+\bar \rho_n^2)dx)$
converges to $(u,\rho, (u_x^2+\bar \rho^2)dx)$ in
$\D$.
\end{lemma}

\begin{proof} The proof follows the same lines as the one of \cite[Proposition 5.1]{HolRay:07}, which we will not repeat here. Here we focus on showing that $\bar r_n\to\bar r\in L^2(\Real)$. Note that we know already from the convergence in Eulerian coordinates that $\vert k_n-k\vert$ converges to $0$.

  We write $g_n=u_{n,x}^2+\bar \rho_n^2$ and
  $g=u_x^2+\bar\rho^2$. Let
  $X_n=(y_n,U_n,h_n,r_n)$ and $X=(y,U,h,r)$ be the
  representatives in $\F_0$ given by
  \eqref{eq:Ldef} of $L(u_n,\rho_n,
  (u_{n,x}^2+\bar\rho^2)dx)$ and
  $L(u,\rho,(u_x^2+\bar\rho^2)dx)$,
  respectively. In particular, we have
\begin{equation}
\label{eq:ydefp}
\int_{-\infty}^{y(\xi)}g(x)\,dx+y(\xi)=\xi\ ,\
\int_{-\infty}^{y_n(\xi)}g_n(x)\,dx+y_n(\xi)=\xi
\end{equation}
and, after taking the difference between the two
equations, we obtain
\begin{equation}
\label{eq:gmingn}
\int_{-\infty}^{y(\xi)}(g-g_n)(x)\,dx+\int_{y_n(\xi)}^{y(\xi)}g_n(x)\,dx+y(\xi)-y_n(\xi)=0.
\end{equation}
Since $g_n$ is positive,
$\abs{y-y_n+\int_{y_n}^y g_n(x)\,d\xi}=\abs{y-y_n}+\abs{\int_{y_n}^y g_n(x)\,d\xi}$
and \eqref{eq:gmingn} implies 
\begin{equation*}
\abs{y(\xi)-y_n(\xi)}\leq\int_{-\infty}^{y(\xi)}\abs{g-g_n}(x)\,dx
\leq\norm{g-g_n}_{L^1}.
\end{equation*}
Since $u_n\to u$ in $H_{0,\infty}(\Real)$ and
$\rho_n\to\rho\in L^2_{\rm const}(\Real)$, also
$g_n\to g$ in $L^1(\Real)$ and it follows that
$\zeta_n\to\zeta\in L^\infty(\Real)$.  The
measures $(u_{x}^2+\bar \rho^2)dx$ and
$(u_{n,x}^2+\bar \rho_n^2)dx$ have, by definition,
no singular part and therefore
\begin{equation}
\label{eq:yxidef2}
y_\xi=\frac{1}{g\circ y+1}
\ \text{ and }\ y_{n,\xi}=\frac{1}{g_n\circ y_n+1}
\end{equation}
almost everywhere. Hence,
\begin{align}
\notag \zeta_{n,\xi}-\zeta_\xi&=(g\circ y-g_n\circ
y_n)y_{n,\xi}y_\xi\\
\label{eq:diffyn}
&=(g\circ y-g\circ
y_n)y_{n,\xi}y_\xi+(g\circ y_n-g_n\circ y_n)y_{n,\xi}y_\xi.
\end{align}
Since $0\leq y_\xi\leq1$, we have
\begin{equation}
\label{eq:gyn1}
\int_\Real\abs{g\circ y_n-g_n\circ y_n}y_{n,\xi}y_\xi\,d\xi\leq\int_\Real\abs{g\circ
y_n-g_n\circ y_n}y_{n,\xi}\,d\xi
=\norm{g-g_n}_{L^1}.
\end{equation}
For any $\epsi>0$, there exists a continuous
function $l$ with compact support such that
$\norm{g-l}_{L^1}\leq\epsi/3$.  We can decompose
the first term on the right-hand side of
\eqref{eq:diffyn} into
\begin{multline}
\label{eq:secdec}
(g\circ y-g\circ y_n)y_{n,\xi}y_\xi=(g\circ y-l\circ y)y_{n,\xi}y_\xi\\+(l\circ y-l\circ
y_n)y_{n,\xi}y_\xi+(l\circ
y_n-g\circ y_n)y_{n,\xi}y_\xi.
\end{multline}
Then, we have
\begin{equation*}
\int_\Real\abs{g\circ y-l\circ
y}y_{n,\xi}y_\xi\,d\xi\leq \int\abs{g\circ
y-l\circ y}y_\xi\,d\xi
=\norm{g-l}_{L^1}\leq\frac{\epsi}{3},
\end{equation*}
and, similarly, we obtain $\int_\Real\abs{g\circ
y_n-l\circ y_n}y_{n,\xi}y_\xi\,d\xi\leq\epsi/3$.
Since $y_n\to y$ in $\Linf$ and $l$ is continuous
with compact support, we obtain by applying the Lebesgue
dominated convergence theorem,  that $l\circ
y_n\to l\circ y$ in $L^1(\Real)$, and thus we can choose $n$
big enough so that
\begin{equation*}
\int_\Real \abs{l\circ y-l\circ
y_n}y_{n,\xi}y_\xi\,d\xi\leq\norm{l\circ y-l\circ
y_n}_{L^1}\leq\frac{\epsi}{3}.
\end{equation*}
Hence, from \eqref{eq:secdec}, we get that
$\int_\Real\abs{g\circ y-g\circ
y_n}y_{n,\xi}y_\xi\,d\xi\leq\epsi$ so that
\begin{equation*}
\lim_{n\to\infty}\int_\Real\abs{g\circ y-g\circ
y_n}y_{n,\xi}y_\xi\,d\xi=0,
\end{equation*} 
and, from \eqref{eq:diffyn} and \eqref{eq:gyn1},
it follows that $\zeta_{n,\xi}\to \zeta_\xi$ in
$L^1(\Real)$. Since $X_n\in\F_0$, $\zeta_{n,\xi}$ is
bounded in $L^\infty(\Real)$,  we finally get that
$\zeta_{n,\xi}\to \zeta_\xi$ in $L^2(\Real)$ and, by
\eqref{eq:Ldef2}, $h_{n}\to h$ in
$L^2(\Real)$. 

We are now ready to show that 
$\bar r_n\to\bar r$ in
$\Ltwo$. By definition we
have $\bar r_n=\bar\rho_n\circ y_ny_{n,\xi}$ and
$\bar r=\bar\rho\circ yy_\xi$, so that 
\begin{align}\label{est:stab2}
 \norm{\bar r_n-\bar r}_{L^2}^2 & = \norm{\bar \rho_n\circ y_ny_{n,\xi}-\bar \rho\circ yy_\xi}_{L^2}^2\\  \nn
& = \int_\Real (\bar\rho_n\circ y_n)^2y_{n,\xi}(y_{n,\xi}-y_\xi)d\xi+\int_\Real \bar\rho_n\circ y_ny_{n,\xi}(\bar\rho_n\circ y_n-\bar\rho_n\circ y)y_\xi d\xi\\ \nn
&\quad +\int_\Real \bar\rho\circ y_ny_{n,\xi}(\bar\rho_n\circ y-\bar\rho\circ y)y_\xi d\xi+\int_\Real (\bar \rho\circ y)^2y_\xi(y_\xi-y_{n,\xi})d\xi\\ \nn
&\quad + \int_\Real \bar\rho\circ yy_\xi(\bar\rho\circ y-\bar \rho\circ y_n)y_{n,\xi}d\xi+\int_\Real \bar \rho\circ yy_\xi(\bar \rho\circ y_n-\bar \rho_n\circ y_n)y_{n,\xi}d\xi.
\end{align}
The first and the fourth term have the same structure, and we therefore only treat the first one. Hence  
\begin{equation}
\norm{(\bar \rho_n\circ y_n)^2y_{n,\xi}( y_{n,\xi}-y_{\xi})}_{L^1}\leq \norm{y_{\xi}-y_{n,\xi}}_{L^1}
\end{equation}
because $(\bar\rho_n\circ y_n)^2y_{n,\xi}\leq h\leq 1$ and thus it tends to $0$ as $n\to\infty$.
In order to investigate the fifth term we will use that $\bar\rho\in L^2(\Real)$ and therefore for any $\varepsilon>0$ there exists a continuous function $\tilde l$ with compact support such that $\norm{\bar\rho-\tilde l}_{L^2}\leq \varepsilon /(3\norm{\rho}_{L^2})$.
Thus we can write
\begin{align*}
 \norm{\bar\rho\circ yy_{\xi}(\bar\rho\circ y-\bar\rho\circ y_n)y_{n,\xi}}_{L^1}&\leq \norm{\bar\rho\circ yy_{\xi}(\bar\rho\circ y-\tilde  l\circ y)y_{n,\xi}}_{L^1}\\ 
& \quad +\norm{\bar\rho\circ yy_{\xi}(\tilde l\circ y-\tilde l\circ y_n)y_{n,\xi}}_{L^1}\\
&\quad +\norm{\bar\rho\circ yy_{\xi}(\tilde l\circ y_n-\bar\rho\circ y_n)y_{n,\xi}}_{L^1} \\ 
& \leq \norm{\bar \rho}_{L^2}(2 \norm{\bar\rho-\tilde l}_{L^2}+\norm{\tilde l\circ y_n-\tilde l\circ y}_{L^2}).
\end{align*}
 Since $y_n\to y\in L^\infty(\Real)$ and $\tilde l$ is continuous with compact support, we obtain by Lebesque's dominated convergence theorem that $\tilde l\circ y_n\to \tilde l\circ y$ in $L^2(\Real)$. In particular, we can choose $n$ big enough so that  $\norm{\bar\rho\circ yy_{\xi}(\bar\rho\circ y-\bar\rho\circ y_n)y_{n,\xi}}_{L^1}\leq \varepsilon$. Since $\varepsilon$ can be chosen arbitrarily small we  obtain  in particular that  
\begin{equation}
 \lim_{n\to\infty}\norm{\bar\rho\circ yy_{\xi}(\bar\rho\circ y-\bar\rho\circ y_n)y_{n,\xi}}_{L^1}=0.
\end{equation}
This immediately implies that also the second term tends to zero using $\bar\rho_n\circ y_n-\bar \rho_n\circ y=(\bar\rho_n\circ y_n-\bar \rho\circ y_n)+(\bar \rho\circ y_n-\bar\rho\circ y)+(\bar\rho\circ y-\bar \rho_n\circ y)$.
As far as the third (and the last) term is concerned, we can conclude as follows
\begin{align*}
 \norm{\bar\rho\circ y_ny_{n,\xi}(\bar\rho_n\circ y-\bar\rho\circ y)y_\xi}_{L^1}&\leq \norm{\bar\rho\circ y_n y_{n,\xi}}_{L^2}\norm{(\bar\rho_n\circ y-\bar\rho\circ y)y_\xi}_{L^2}\\ &\leq \norm{\bar\rho}_{L^2}\norm{\bar\rho_n-\bar\rho}_{L^2},
\end{align*}
which again tends to zero since by assumption $\bar\rho_n\to\bar\rho\in L^2(\Real)$.
Hence all terms in \eqref{est:stab2} tend to $0$ as $n\to\infty$ and therefore $\bar r_n\to\bar r\in L^2(\Real)$.

Finally we want to point out that when proving $U_{n,\xi}\to U_\xi$ in $L^2(\Real)$, one has that 
\begin{equation}
\label{eq:unxirew}
U_{n,\xi}^2=h_{n}-h_{n}^2-\bar r_n^2,
\end{equation}
and a corresponding identity holds
for $U_\xi$ which allows us to conclude as in \cite[Proposition 5.1]{HolRay:07}. 

\end{proof}

\begin{lemma}\label{thm:top2}
 Let $(u_n,\rho_n,\mu_n)$ be a sequence in $\D$ that converges to $(u,\rho,\mu)$ in $\D$. Then 
\begin{equation*}
  u_n\rightarrow u \text{ in } L^\infty(\Real),\quad \bar \rho_n \overset{\ast}{\rightharpoonup}\bar \rho,\quad k_n \to k\in \Real, \quad  \text{ and }\quad  \mu_n \overset{\ast}{\rightharpoonup}\mu.
\end{equation*}
\end{lemma}

\begin{proof}
  We denote by $X_n=(y_n,U_n,h_n,r_n)$ and
  $X=(y,U,h,r)$ the representative of
  $L(u_n,\rho_n,\mu_n)$ and $L(u,\rho,\mu)$ given
  by \eqref{eq:Ldef}. For any $x\in\Real$, there
  exist $\xi_n$ and $\xi$, not necessarily unique,
  such that $x=y_n(\xi_n)$ and $x=y(\xi)$. We set
  $x_n=y_n(\xi)$. We have
\begin{equation}
\label{eq:udec}
u_n(x)-u(x)=u_n(x)-u_n(x_n)+U_n(\xi)-U(\xi)
\end{equation}
and
\begin{align}
\notag
\abs{u_n(x)-u_n(x_n)}&=\abs{\int_{\xi}^{\xi_n}U_{n,\xi}(\eta)\,d\eta}\\
\notag
&\leq\sqrt{\xi_n-\xi}\left(\int_{\xi}^{\xi_n}U_{n,\xi}^2\,d\eta\right)^{1/2}&\text{(Cauchy--Schwarz)}\\
\notag
&\leq\sqrt{\xi_n-\xi}\left(\int_{\xi}^{\xi_n}y_{n,\xi}h_{n}\,d\eta\right)^{1/2}&\text{(from
\eqref{eq:lagcoord3})}\\ \notag
&\leq\sqrt{\xi_n-\xi}\sqrt{\abs{y_n(\xi_n)-y_n(\xi)}}&\text{(since
$h_{n}\leq1$)}\\ \notag
&=\sqrt{\xi_n-\xi}\sqrt{y(\xi)-y_n(\xi)}\\
\label{eq:undif}
&\leq \sqrt{\xi_n-\xi}\norm{y-y_n}_{L^\infty}^{1/2}.
\end{align}
From $\vert y(\xi)-\xi\vert\leq\mu(\Real)$, we get 
\begin{equation*}
\abs{\xi_n-\xi}\leq
2\mu_n(\Real)+\abs{y_n(\xi_n)-y_n(\xi)}
=2\norm{h_n}_{L^1}+\abs{y(\xi)-y_n(\xi)}
\end{equation*}
and, therefore, since $h_n\to h$ in $L^1(\Real)$ (because $h=U_\xi^2+\bar r^2-h\zeta_\xi$) and $y_n\to y$ in
$\Linf$, $\abs{\xi_n-\xi}$ is bounded by a
constant $C$ independent of $n$. Then,
\eqref{eq:undif} implies
\begin{equation}
\label{eq:undif2}
\abs{u_n(x)-u_n(x_n)}\leq C
\norm{y-y_n}_{L^\infty}^{1/2}. 
\end{equation}
Since $y_n\to y$ and $U_n\to U$ in $\Linf$, it
follows from \eqref{eq:udec} and \eqref{eq:undif2}
that $u_n\to u$ in $\Linf$. By weak-star
convergence, we mean that
\begin{equation}
\label{eq:wsconv}
\lim_{n\to\infty}\int_\Real\bar \rho_n\phi\,dx=\int_\Real\bar\rho \phi\,dx
\end{equation}
for all continuous functions with compact
support. It follows from \eqref{eq:umudef2} that
\begin{equation}
\label{eq:compfmu}
\int_\Real\bar\rho_n\phi\,dx=\int_\Real\bar r_n\phi\circ y_n\,d\xi\ \text{ and }\
\int_\Real\bar \rho \phi\,dx=\int_\Real\bar r \phi\circ y
\,d\xi.
\end{equation}
Since
$y_n\to y$ in $\Linf$, the support of $\phi\circ
y_n$ is contained in some compact which can be
chosen independently of $n$ and, from Lebesgue's
dominated convergence theorem, we have that
$\phi\circ y_n\to\phi\circ y$ in $\Ltwo$. Hence,
since $\bar r_{n}\to \bar r$ in $\Ltwo$,
\begin{equation*}
\lim_{n\to\infty}\int_\Real \phi\circ y_n
\bar r_{n}\,d\xi=\int_\Real\phi\circ y \bar r\,d\xi,
\end{equation*}
and \eqref{eq:wsconv} follows from
\eqref{eq:compfmu}
Similarly one can show that $\mu_n$ converges weakly to $\mu$.
Finally, that $k_n$ converges to $k$ is obvious by the definition of convergence in Eulerian coordinates.
\end{proof}

\bibliographystyle{plain}

\begin{thebibliography}{99}

\bibitem{BreCons:07}
A. Bressan and A. Constantin.
\newblock Global conservative solutions of the {C}amassa--{H}olm equation.
\newblock {\em Arch. Ration. Mech. Anal.}, 183(2):215--239, 2007.

\bibitem{BreCons:09}
A. Bressan and A. Constantin.
\newblock Global dissipative solutions of the Camassa--Holm equation. 
\newblock {\em Analysis and Applications}, 5:1--27, 2007.


\bibitem{ChenLiu2010}
R. M. Chen and Y. Liu.
\newblock Wave breaking and global existence for a generalized two-component Camassa--Holm system.
\newblock Inter. Math Research Notices, Article ID rnq118, 36 pages, 2010.

\bibitem{ChenLiuZhang} 
M. Chen,, S.-Q. Liu,  and Y.  Zhang.
\newblock A two-component generalization of the Camassa--Holm equation and its solutions.
\newblock  {\em Lett.  Math. Phys.},  75:1--15, 2006.

\bibitem{MR2474608}
A. Constantin and R.~I. Ivanov.
\newblock On an integrable two-component {C}amassa--{H}olm shallow water
  system.
\newblock {\em Phys. Lett. A}, 372(48):7129--7132, 2008.

\bibitem{eschlechyin:07} 
J. Escher, O. Lechtenfeld, and Z. Yin. 
\newblock Well-posedness and blow-up phenomena for the 2-component
  {C}amassa--{H}olm equation.
\newblock {\em Discrete Contin. Dyn. Syst.}, 19(3):493--513, 2007.

\bibitem{FuQu:09}
Y. Fu and C. Qu.
\newblock Well posedness and blow-up solution for a new coupled 
  {C}amassa--{H}olm equations with peakons.
\newblock {\em J. Math. Phys.}, 50:012906, 2009.

\bibitem{GHRb:10} K. Grunert, H. Holden, and
  X. Raynaud.  \newblock Lipschitz metric for the
  {C}amassa--{H}olm equation on the
  line. \newblock {\em Submitted} 2010, \newblock
  \arxiv{1010.0561}.

\bibitem{GHR:12} K. Grunert, H.  Holden, and X.
  Raynaud.  \newblock Global conservative
  solutions of the {C}amassa--{H}olm equation for
  initial data with nonvanishing asymptotics. \newblock {\em
    Submitted} 2011, \newblock \arxiv{1106.4125}.
	
\bibitem{GHR:11}
K. Grunert, H.  Holden, and X.   Raynaud.
\newblock Lipschitz metric for the periodic {C}amassa--{H}olm equation.
\newblock {\em J. Differential Equations}, 250(3):1460--1492, 2011.

\bibitem{GuanKarlsenYin2010} 
C. Guan, K. H. Karlsen, and Z. Yin.
\newblock Well-posedness and blow-up phenomena for a modified  two-component Camassa--Holm equation.
\newblock In: {\em Nonlinear Partial Differential Equations and Hyperbolic Wave Phenomena}�  (eds. H. Holden and K. H. Karlsen), Amer. Math. Soc., Providence, pp. 199--220, 2010.

\bibitem{GuanYin2010} 
C. Guan and Z. Yin.
\newblock Global weak solutions for a modified two-component Camassa--Holm equation.
\newblock {\em Ann. I. H. Poincar{\'e} -- AN} 28:623--641, 2011.

\bibitem{GuanYin2010a} 
C. Guan and Z. Yin.
\newblock Global existence and blow-up phenomena for an integrable  two-component Camassa--Holm shallow water system.
\newblock {\em J. Differential Equations},  248:2003--2014, 2010.

\bibitem{GuiLiu2011}  
G. Gui and Y. Liu.
\newblock On the Cauchy problem for the two-component Camassa--Holm system.
\newblock {\em Math Z},  268:45--66, 2011.

\bibitem{GuiLiu2010}  
G. Gui and Y. Liu.
\newblock On the global existence and wave breaking criteria for the two-component Camassa--Holm system.
\newblock {\em J. Func. Anal.},  258:4251--4278, 2010.

\bibitem{GuoZhou2010}
Z. Guo and Y. Zhou.
\newblock On solutions to a two-component generalized Camassa--Holm equation.
\newblock {\em Studies Appl. Math.}, 124:307--322, 2010.

\bibitem{henry:09}
D. Henry.
\newblock Infinite propagation speed for a two component {C}amassa--{H}olm
  equation.
\newblock {\em Discrete Contin. Dyn. Syst. Ser. B}, 12(3):597--606, 2009.

\bibitem{HolRay:07}
H. Holden and X. Raynaud.
\newblock Global conservative solutions of the {C}amassa--{H}olm equation---a
  {L}agrangian point of view.
\newblock {\em Comm. Partial Differential Equations}, 32(10-12):1511--1549,
  2007.

\bibitem{HolRay:09}
H. Holden and X. Raynaud.
\newblock Dissipative solutions for the Camassa--Holm equation. 
\newblock {\em Discrete Contin. Dyn. Syst.} 24:1047--1112, 2009.

\bibitem{Kuzmin} 
P. A.  Kuz'min.
\newblock Two-component generalizations of the Camassa--Holm equation.
\newblock {\em Math. Notes},  81:130--134, 2007.

\bibitem{Mohajer} 
K. Mohajer.
\newblock A note on traveling wave solutions to the Camassa--Holm equation.  
\newblock {\em J. Nonlin. Math. Phys.},  16:117--125, 2009.

\bibitem{Mustafa} 
O. G. Mustafa.
\newblock On smooth traveling waves of an integrable two-component Camassa--Holm shallow water system.
\newblock {\em Wave Motion}, 46:397--402, 2009.

\bibitem{OlverRosenau}
P.~J. Olver and P. Rosenau.
\newblock Tri-hamiltonian duality between solitons and solitary-wave solutions
  having compact support.
\newblock {\em Phys. Rev. B}, 53(2):1900--1906, 1996.

\bibitem{TanYin} 
W. Tan and Z. Yin.
\newblock Global dissipative solutions of a modified two-component
Camassa--Holm shallow water system.
\newblock {\em J. Math. Phys.}, 52:033507, 2011.

\bibitem{WangHuangChen}  
Y. Wang,  J. Huang, and L. Chen.
\newblock Global conservative solutions of the two-component Camassa--Holm shallow water system.
\newblock {\em Int.  J.  Nonlin. Science},  9:379--384, 2009.

\end{thebibliography}

\end{document}